\documentclass[oneside,english]{amsart}
\usepackage[T1]{fontenc}
\usepackage[utf8]{luainputenc}
\usepackage{color}
\usepackage{babel}
\usepackage{refstyle}
\usepackage{amsthm}
\usepackage{amssymb}

\usepackage[all]{xy}
\usepackage[unicode=true,pdfusetitle,
 bookmarks=true,bookmarksnumbered=false,bookmarksopen=false,
 breaklinks=false,pdfborder={0 0 1},backref=page,colorlinks=true]
 {hyperref}

\makeatletter

\AtBeginDocument{\providecommand\secref[1]{\ref{sec:#1}}}
\AtBeginDocument{}
\AtBeginDocument{}
\AtBeginDocument{}
\AtBeginDocument{}
\AtBeginDocument{\providecommand\eqref[1]{\ref{eq:#1}}}
\RS@ifundefined{subref}
  {\def\RSsubtxt{section~}\newref{sub}{name = \RSsubtxt}}
  {}
\RS@ifundefined{thmref}
  {\def\RSthmtxt{theorem~}\newref{thm}{name = \RSthmtxt}}
  {}
\RS@ifundefined{lemref}
  {\def\RSlemtxt{lemma~}\newref{lem}{name = \RSlemtxt}}
  {}

\numberwithin{equation}{section}
\numberwithin{figure}{section}
  \theoremstyle{plain}
  \newtheorem*{thm*}{\protect\theoremname}
\theoremstyle{plain}
\newtheorem{thm}{\protect\theoremname}[section]
  \theoremstyle{plain}
  \newtheorem{prop}[thm]{\protect\propositionname}
  \theoremstyle{definition}
  \newtheorem{defn}[thm]{\protect\definitionname}
  \theoremstyle{plain}
  \newtheorem{lem}[thm]{\protect\lemmaname}
  \theoremstyle{plain}
  \newtheorem{cor}[thm]{\protect\corollaryname}
  \theoremstyle{remark}
  \newtheorem{rem}[thm]{\protect\remarkname}
  \theoremstyle{claim}
  \newtheorem{claim}[thm]{\protect\claimname}
\theoremstyle{note}
  \newtheorem{note}[thm]{\protect\notename}
\theoremstyle{acknowledgment}
  \newtheorem{acknowledgment}[thm]{\protect\acknowledgmentname}
\makeatother

  \providecommand{\corollaryname}{Corollary}
  \providecommand{\definitionname}{Definition}
  \providecommand{\lemmaname}{Lemma}
  \providecommand{\claimname}{Claim}
  \providecommand{\propositionname}{Proposition}
  \providecommand{\remarkname}{Remark}
  \providecommand{\theoremname}{Theorem}
\providecommand{\theoremname}{Theorem}
\providecommand{\notename}{Note}
\providecommand{\acknowledgmentname}{Acknowledgment}
\begin{document}

\title{verbally prime T-ideals and graded division algebras }

\author{Eli Aljadeff \and Yaakov karasik}

\address{Department of Mathematics, Technion - Israel Institute of Technology,
Haifa 32000, Israel}

\email{aljadeff 'at' tx.technion.ac.il (E. Aljadeff),}

\email{yaakov 'at' tx.technion.ac.il (Y. Karasik).}

\begin{abstract}

Let $F$ be an algebraically closed field of characteristic zero and let $G$ be a finite group. We consider graded Verbally prime $T$-ideals in the free $G$-graded algebra. It turns out that equivalent definitions in the ordinary case (i.e. ungraded) extend to nonequivalent definitions in the graded case, namely verbally prime $G$-graded $T$-ideals and strongly verbally prime $T$-ideals. At first, following Kemer's ideas, we classify $G$-graded verbally prime $T$-ideals. The main bulk of the paper is devoted to the stronger notion. We classify $G$-graded strongly verbally prime $T$-ideals which are $T$-ideal of affine $G$-graded algebras or equivalently $G$-graded $T$-ideals that contain a Capelli polynomial. It turns out that these are precisely the $T$-ideal of $G$-graded identities of finite dimensional $G$-graded, central over $F$ (i.e. $Z(A)_{e}=F$) which admit a $G$-graded division algebra twisted form over a field $k$ which contains $F$ or equivalently over a field $k$ which contains enough roots of unity (e.g. a primitive $n$-root of unity where $n = ord(G)$).

\end{abstract}

\maketitle

\section{introduction}

Let $F$ be an algebraically closed field of characteristic zero. The study of PI algebras over $F$ is closely related to the study of $T$-ideals
of the noncommutative free associative algebra $F\left\langle X\right\rangle$ where $X$ is a countable set of variables.
Since any $T$-ideal $\Gamma$ appears as the ideal of polynomial identities of a suitable algebra $A$, it is natural to consider special types of $T$-ideals and try to understand
what algebras have them as their ideal of identities. In the ungraded
case two types of \textit{primeness} on $T$-ideals are considered, namely \textit{prime} $T$-ideals and \textit{verbally prime} $T$-ideals. By definition a $T$-ideal $I$ is prime if for any two ideals $J_{1}, J_{2}$ in $F\langle X \rangle$ such that $J_{1}J_{2} \subseteq I$, either $J_{1} \subseteq I$ or $J_{2} \subseteq I$ whereas $I$ is verbally prime if a similar condition holds whenever the ideals $J_{1}$ and $J_{2}$ are $T$-ideals.

We adopt the definition from $T$-ideals and say that an algebra $A$ is \textit{verbally prime} if its $T$-ideal of identities $Id(A)$ is verbally prime. As for primeness, we recall that an algebra $A$ is prime if for ideals $I$ and $J$ in $A$ with $IJ=0$ we must have $I=0$ or $J=0$, hence we cannot just define primeness in terms of the corresponding $T$-ideal of identities. Nevertheless, by structure theorems of Amitsur (\cite{Amitsur}) and Posner (\cite{GiaZai}, section 1.11) we have

 \begin{thm} \label{basic prime theorem ungraded}
Let A be a PI algebra. Then it is PI equivalent to a prime algebra B if and only if its T-ideal of polynomial identities Id(A) is a prime T-ideal. Furthermore, $A$ is a prime PI algebra if and only if it is PI equivalent to $M_{n}(F)$ for some $n > 0$.
\end{thm}

As for the classification of verbally prime PI algebras recall the following fundamental result which is due to Kemer \cite{Kemer1984}.

\begin{thm} \label{basic theorem ungraded}

$A$ is a verbally prime PI algebra if and only if it is PI equivalent to the Grassmann envelope of a finite dimensional
$\mathbb{Z}/2\mathbb{Z}-$simple $F$-algebra $A$. It follows that if $A$ is also
affine then it is PI equivalent to a prime PI algebra.
\end{thm}

Posner's Theorem was recently extended to the $G$-graded setting by Karasik (see \cite{Karasik}). As a corollary he obtains a generalization of Theorem \ref{basic prime theorem ungraded} for $G$-graded algebras (see Theorem \ref{classification of G-graded prime} below).

Our goal in this paper is to extend Theorem \ref{basic theorem ungraded} to the $G$-graded setting where $G$ is an arbitrary finite group. Somewhat more precisely, we wish to classify $G-$graded algebras (up to graded PI equivalence) which are PI as ungraded algebras and whose $T$-ideal of $G$-graded identities satisfy conditions which are analogue to verbally prime (see below for the precise definitions).

\begin{rem}
A key ingredient in Kemer's proof of Theorem \ref{basic theorem ungraded} is his own celebrated Representability Theorem
which states that every PI algebra is PI equivalent to the Grassmann
envelope of a finite dimensional $\mathbb{Z}/2\mathbb{Z}-$ graded algebra.
Also in the $G$-graded setting an important tool in our proofs is the Representability Theorem for $G$-graded algebras (see Theorem (A) below or (\cite{AB}, Theorems 1.1 and 1.3)). Since the Representability Theorem makes sense only for $G$-graded algebras that are ungraded PI, we restrict our discussion to algebras of that type.

\end{rem}

It may seem that the generalization to $G-$ graded algebras is a natural extension of Kemer's original work. Surprisingly, this is not the case as we shall soon explain. It is well known that for ordinary $T$-ideals (i.e. ungraded) the property of a $T$-ideal $I$ being verbally prime may be expressed also in terms of polynomials, that is, $I$ is verbally prime if and only if the following holds: if $f,g \in F\left\langle X\right\rangle$ are polynomials defined on disjoint sets of variables with $fg \in I$, then either $f$ or $g$ is in $I$. Note that if $I$ is the $T$-ideal of identities of an algebra $A$, then the above condition on $f$ and $g$ says that if $fg$ is an identity of $A$ then either $f$ or $g$ is an identity of $A$. Most of the time we shall work with that terminology.

It turns out that the equivalence of the two conditions mentioned above for verbally primeness of $T$-ideals is not longer valid when considering $G-$graded $T$-ideals and hence we obtain two \textit{primeness} conditions
on $G-$graded $T$-ideals. One finds that the corresponding condition for $G$-graded $T$-ideals expressed in terms of polynomials is \textit{strictly stronger} and hence we will say that $G$-graded $T$-ideals (or $G$-graded algebras) satisfying that property are \textit{strongly verbally prime}. For future reference we record below the following definitions.

\begin{defn}
Let $\Gamma$ be a $G$-graded $T$-ideal which contains a nonzero PI.

\begin{enumerate}

\item

We say $\Gamma$ is \textit{$G$-graded prime} if for any $G$-graded ideals $I$ and $J$ with $IJ \subseteq \Gamma$, then
 $I \subseteq \Gamma$ or $J \subseteq \Gamma$.

\item
We say $\Gamma$ is \textit{$G$-graded verbally prime} if for any $G$-graded $T$-ideals $I$ and $J$ with $IJ \subseteq \Gamma$, then
 $I \subseteq \Gamma$ or $J \subseteq \Gamma$.

\item

We say $\Gamma$ is \textit{$G$-graded strongly verbally prime} if for any $G$-graded homogeneous polynomials $f$ and $g$ defined on \textit{disjoint} sets of variables with $fg \in \Gamma$, then $f \in \Gamma$ or $g \in \Gamma$.
\end{enumerate}

\end{defn}

As mentioned above, Posner's theorem was generalized to the $G$-graded setting by Karasik (\cite{Karasik}, Corollary 4.7). In particular it gives the classification of $G$-prime $T$-ideals.

\begin{thm}\label{classification of G-graded prime}
Let $A$ be a $G$-graded algebra which is PI as an ungraded algebra. Then its $T$-ideal of identities is prime if and only if $A$ is $G$-graded PI equivalent to a finite dimensional $G$-simple $F$-algebra.

\end{thm}

The classification of $G$-graded verbally prime $T$-ideals (or algebras) was obtained by Berele and Bergen in \cite{BerBergen} and independently by Di Vincenzo in \cite{Di Vincenzo} for the important case where $G=\mathbb{Z}_{2}$. The precise result is as follows.

\begin{thm}\label{classification of G-graded verbally prime}
Let $A$ be a $G-$graded algebra which is also PI as an ungraded
algebra. Then $A$ is verbally prime if and only if it is $G-$graded
PI equivalent to the Grassmann envelope of a finite dimensional $G_{2}=\mathbb{Z}/2\mathbb{Z}\times G-$simple $F$-algebra. If $A$ is also affine, then $A$ is $G$-verbally
prime if and only if it is $G$-prime, that is, if and only if $A$ is  $G-$graded PI equivalent to a finite dimensional $G-$simple $F$-algebra.
\end{thm}

The classification of $G$-graded $T$-ideals satisfying the stronger condition, namely $G$-\textit{strongly verbally prime}, appears to be more involved. One needs to translate the language of $G-$graded polynomial identities into the algebraic structure of $G$-simple algebras. In this paper we provide a full answer in the affine case, namely we will classify up to $G-$graded PI equivalence $G-$strongly verbally prime algebras which are affine. At this point, the nonaffine case remains open.

From Theorem \ref{classification of G-graded verbally prime} and the fact that \textit{strongly verbally primeness} implies \textit{verbally primeness} we have that in order to classify the $G$-graded affine algebras which are strongly verbally prime we need to determine the finite dimensional $G$-simple algebras $A$ which satisfy the following condition:
if $f$ and $g$ are $G$-graded homogeneous polynomials (with respect to the $G$-grading) with disjoint sets of graded variables and such that $fg$ is an identity of $A$, then $f$ or $g$ is an identity of $A$. In order to state our main results we recall the structure of finite dimensional $G$-graded algebras which are $G$-simple. This result is key and is due to Bahturin, Sehgal and Zaicev \cite{BSZ}.

\begin{thm}
Let $A$ be a f.d. $G-$simple algebra over an algebraically closed field
$F$ of zero characteristic. Then there is a subgroup $H$ of $G$, a two cocycle $c\in Z^{2}(H,F^{*})$
$($the $H-$action on $F^{*}$ is trivial$)$ and $\mathfrak{g}=(g_{1},\ldots,g_{m}) \in G ^{(m)}$,
such that $A\cong F^{c}H\otimes M_{m}(F)$, where $($by this identification$)$ the $G-$grading on $A$ is given
by
\[
A_{g}=\mbox{span}_{F}\left\{ u_{h}\otimes e_{i,j}|g=g_{i}^{-1}hg_{j}\right\} .
\]
\end{thm}

It is easy to see that up to a $G$-graded isomorphism the algebra $A$ does not depend on the $2$-cocycle $c$ but only on its cohomology class $[c] \in H^{2}(H,F^*)$.
We therefore denote the $G$-grading on the algebra $A$ by $P_{A}=(H,[c],\mathfrak{g})$. We also say that $(H,[c],\mathfrak{g})$ is a presentation of the $G$-graded algebra $A$.

Note that if an algebra $A$ is $G$-graded then it is $U$-graded in a natural way whenever $G$ is a subgroup of $U$ (just by putting $A_{u}=0$ for $u \in U - G$). It is clear in that case that $A$ is strongly verbally prime as a $G$-graded algebra if and only if $A$ is strongly verbally prime as a $U$-graded algebra. We therefore assume for the rest of the paper that the grading is \textit{connected}. This means that the support of the grading generates the group $G$.
Here is one of the two main results of the paper.

\begin{thm} \label{main theorem1}
Let $\Gamma$ be a $G$-graded $T$-ideal where $G$ is a finite group. Suppose $\Gamma$ contains an ungraded Capelli polynomial $c_{n}$ for some integer $n$ (equivalently, $\Gamma$ is the $T$-ideal of $G$-graded identities of an affine PI algebra). Then the ideal $\Gamma$ is $G$-graded strongly verbally prime if and only if it is the $T$-ideal of $G$-graded identities of a finite dimensional $G$-simple algebra $A$ with presentation $P_{A}=(H, \alpha, \mathfrak{g}=(g_1,\ldots,g_k))$ satisfying the following conditions:

\begin{enumerate}
\item
The group $H$ is normal in $G$.

\item

The different cosets of $H$ in $G$ are equally represented in the $k$-tuple $\mathfrak{g}=(g_1,\ldots,g_{k})$. In particular the integer $k$ is a multiple of $[G:H]$.

\item

The cohomology class $\alpha \in H^{2}(H,F^{*})$ is $G$-invariant. Here $G$ acts on $H^{2}(H,F^{*})$ via conjugation on $H$ and trivially on $F^{*}$.

\end{enumerate}

\end{thm}
The following special cases are of interest.

\begin{cor}
Let $A$ be a $G$-simple algebra with presentation $P_{A}=(H,\alpha,\mathfrak{g})$. Suppose $H$ is normal and $H$-cosets representatives are equally represented in $\mathfrak{g}$. If $\alpha$ is trivial then $A$ is strongly verbally prime.

\end{cor}

\begin{cor}
Suppose $G$ is finite abelian and let $A$ be a $G$-simple algebra with presentation $P_{A}=(H,\alpha,\mathfrak{g})$. Then $A$ is strongly verbally prime if and only if the $H$-cosets representatives are equally represented in $\mathfrak{g}$.
\end{cor}

In particular we have

\begin{cor}
Let $A$ be a finite dimensional $\mathbb{Z}_{2}=\{e, \sigma\}$-simple algebra over $F$ with presentation $P_{A}=(H,\alpha,\mathfrak{g})$. Then $A$ is strongly verbally prime if and only if $H \cong \mathbb{Z}_{2}$ or else $H = \{e\}$ and $\mathfrak{g}= (e,\ldots,e, \sigma, \ldots, \sigma)$ where $e$ and $\sigma$ have the same multiplicity.
\end{cor}

One of the main features of finite dimensional strongly verbally prime algebras is its intimate relation to finite dimensional $G$-graded division algebras, that is finite dimensional $G$-graded algebras whose nonzero homogeneous elements are invertible (e.g. the group algebra $FG$). Naturally, these algebras will be defined over fields $K$ which are usually nonalgebraically closed. Nonetheless, for our purposes, we'll only consider fields which contain $F$, and hence any algebra we consider in this paper is by default an $F$-algebra.



Let $A$ be a finite dimensional $G$-simple $F$-algebra and suppose further that $Z(A)_{e} = F1_{A}$ where $e$ is the identity element of $G$. It is easy to see that if $A$ is a $G$-graded division algebra, then necessarily $A$ is $G$-graded strongly verbally prime. Considerably more generally, if an algebra $A$ admits a $G$-graded $K$-form $B_{K}$ (in the sense of \textit{descent}) for which every nonzero homogeneous element is invertible then again $A$ is strongly verbally prime. We say that $A$ \textit{has a $G$-graded $K$-form division algebra}.

Using Theorem \ref{main theorem1} we obtain the second main result of the paper. Notation as above.

\begin{thm}\label{main theorem2}
Let $A$ be a finite dimensional algebra $G-$simple $F$-algebra where $Z(A)_{e} = F1_{A}$. Then $A$ is strongly verbally prime if and only if it has a $G$-graded division algebra form. That is, there exists a finite dimensional $G$-graded division algebra $B$ over $K$ where $Z(B)_{e} = K1_{B}$, and a field $E$ which extends $K$ (and $F$) such that $B_{E}$ and $A_{E}$ are $G$-graded isomorphic algebras over $E$.
\end{thm}

Combining Theorems \ref{main theorem1} and \ref{main theorem2} we obtain the main corollary of the paper.

\begin{cor}
Suppose $F$ is an algebraically closed field of characteristic zero and let $A$ be a finite dimensional $G$-graded $F$-algebra. Then $A$ admits a $G$-graded $K$-form division algebra $B$, some field $K$, if and only if $A$ is $G$-simple whose presentation $P_{A}=(H, \alpha, \mathfrak{g}=(g_1,\ldots,g_k))$ satisfies conditions $(1)-(3)$ in Theorem \ref{main theorem1}.
\end{cor}

Once again, it is immediate that if $A$ is finite dimensional and has a $G$-graded $K$-form division algebra then $A$ is strongly verbally prime (i.e. if $f$ and $g$ are $G$-homogeneous polynomials on disjoint sets of variables such that $fg \in Id_{G}(A)$ then either $f$ or $g$ is in $Id_{G}(A)$). However, it turns out that in that case one can remove the condition on the disjointness of sets of the variables. In fact we have

\begin{thm} \label{nondisjoin variables}
Notation as above. Let $A$ be a $G-$graded $F$-algebra (possibly infinite dimensional). Then $A$ is $G-$PI equivalent to a finite dimensional $G-$division algebra over some field $K$ if and only if for any two $G$-homogeneous polynomials $f$ and $g$ such that $fg \in Id_{G}(A)$ then either $f \in Id_{G}(A)$ or $g \in Id_{G}(A)$. Consequently, any of these conditions holds if and only if $A$ is $G-$strongly verbally prime and $G-$PI equivalent to a finite dimensional algebra.

\end{thm}

Here is the outline of the paper. In \secref{Preliminaries} we introduce
some basic results and notation in graded PI theory. As mentioned earlier, Theorem \ref{classification of G-graded verbally prime} appears in \cite{BerBergen} although there is some confusion with the different definitions of verbally primeness. For completeness of this article we devote a short section (\secref{Verbally-prime}) for its proof.

In \secref{strongly-verbally-prime} we discuss the strongly verbally prime condition and show the conditions of Theorem \ref{main theorem1} are necessary. In the same section we discuss also their sufficiency however the proof of Theorem \ref{main theorem1} is completed only in \secref{division-algebras} after we discuss $G$-graded division algebra twisted forms and prove Theorems \ref{main theorem2} and \ref{nondisjoin variables}.

\section{\label{sec:Preliminaries}Preliminaries}

\subsection{$G-$graded PI}

We fix for the rest of the article a finite group $G$ and an algebraically
closed field of characteristic zero $F$. We are considering algebras $A$
having a $G-$graded structure, that is $A=\oplus_{g\in G}A_{g}$,
where $A_{g}$ is an $F-$vector space contained in $A$, and $A_{g}A_{h}\subseteq A_{gh}$
for every $g,h\in G$. An element in the set $\cup_{g\in G}A_{g}$
is said to be \emph{homogenous}. Sometimes we will say that $a\in A_{g}$
is an element of \emph{degree} $g$ and denote $\mbox{deg}a=g$.

A $G-$graded polynomial is an element of the noncommutative free
algebra $F\left\langle X_{G}\right\rangle$ generated by the variables
of $X_{G}=\{x_{i,g}|i\in\mathbb{N},\, g\in G\}$. It is clear that
$F\left\langle X_{G}\right\rangle $ is itself $G-$graded, where
$F\left\langle X_{G}\right\rangle _{g}$ is the span of all monomials
$x_{i_{1},g_{1}}\cdots x_{i_{n},g_{n}}$, where $g_{1}\cdots g_{n}=g$.

For elements in $F\left\langle X_{G}\right\rangle$ we adopt the following terminology. A $G$-graded polynomial $p$ is said to be $G$-\textit{homogeneous} if it belongs to $F\left\langle X_{G}\right\rangle_{g}$ for some $g \in G$. Moreover, we say $p$ is \textit{multihomogeneous } if any variable $x_{i_g}$ appears the same number of times in each monomial of $p$. More restrictive, we say $p$ is \textit{multilinear} if it is multihomogeneous and any variable $x_{i_g}$ appears at most once in each monomial of $p$. Note that unless $G$ is abelian, multihomogeneous or multilinear polynomials are not necessarily $G$-homogeneous.

We say that a $G-$graded polynomial $f(x_{i_{1},g_{1}},\ldots,x_{i_{n},g_{n}})$
is a $G-$graded identity of a $G-$graded algebra $A$ if $f(a_{1},\ldots,a_{n})=0$
for all $a_{1}\in A_{g_{1}},\ldots,a_{n}\in A_{g_{n}}$. The set of all
such polynomials is denoted by $Id_{G}(A)$ and it is evident that
it is an ideal of $F\left\langle X_{G}\right\rangle $ which is $G$-graded. In fact it
is a $G-$graded T-ideal, that is, an ideal which is closed under $G-$graded
endomorphisms of $F\left\langle X_{G}\right\rangle $.

It is often convenient to view any ungraded polynomial as $G-$graded
via the embedding
\[
F\left\langle X\right\rangle \to F\left\langle X_{G}\right\rangle ,
\]
where $x_{i}\to\sum_{g\in G}x_{i,g}$. This identification respects
the $G-$graded identities in the sense that if $f\in Id(A)$ (the
$T-$ideal of ungraded identities of $A$) then $f\in Id_{G}(A)$.

As in the ungraded case, the set of $G-$homogeneous multilinear identities
of $A$
\[
\left\{ \sum_{\sigma\in \Lambda_{n,g}\subseteq S_{n}}\alpha_{\sigma}x_{i_{\sigma(1)},g_{\sigma(1)}}\cdots x_{i_{\sigma(n)},g_{\sigma(n)}}\in Id_{G}(A)|\,g \in G, \alpha_{\sigma}\in F,\,i_{1},\ldots,i_{n}\mbox{ are all distinct}\right\}
\]
T ($G-$)generates $Id_{G}(A)$. Here, $\Lambda_{n,g}$ denotes the set of permutations which yield monomials of $G$-degree $g$.

Let $f\in F\left\langle X_{G}\right\rangle $, $A$ a $G-$graded
algebra and $S\subseteq\cup_{g\in G}A_{g}$. We write $f(S)$ for
the subset of $A$ consisting of all graded evaluations of $f$
on elements of $S$. More generally, if $S \subseteq A$ (i.e. not necessarily homogeneous elements) we set $f(S):= f(S\cap (\cup_{g\in G}A_{g}))$.
Note that if $f$ is $G$-graded multilinear then $f$ is a graded identity of $A$ if and only if $f(S)=0$ whenever $S$ contains a basis of $A$ consisting of homogeneous elements.

\subsection{Grassmann envelope}

Let $V$ be a countably generated $F-$vector space. Denote by $T(V)$ the unital
tensor $F-$algebra of $V$. That is
\[
T(V)=F\oplus\bigoplus_{k=1}^{\infty}V^{\otimes k}
\]
with the obvious addition and multiplication. The unital Grassmann algebra
$E$ is defined to be $E=T(V)/\langle T_{0}(V)^{2}\rangle$, where $\langle S \rangle$ is the ideal generated by $S$ and

\[
T_{0}(V)=\bigoplus_{k=1}^{\infty}V^{\otimes k}.
\]
 In other words, if $e_{1},e_{2},\ldots$ is a basis of $V$, we may
see $E$ as the unital $F-$algebra generated by $1, e_{1},e_{2},...$
and satisfies the relations $e_{i}e_{j}=-e_{j}e_{i}$ (so $e_{i}^{2}=0$).
In fact, the elements $a_{i_{1},..,i_{k}}:=e_{i_{1}}\cdots e_{i_{k}}$, where $k=0,1,...$
and $i_{1}<i_{2}<\cdots<i_{k}\in\mathbb{N}$, form a basis of $E$ (for
$k=0$ we just mean $a=1$).

It is easy to see that $E$ is $C_{2}=\mathbb{Z}/2\mathbb{Z}-$graded:
$E=E_{0}\oplus E_{1}$, where $E_{0}$ is spanned by elements $a_{i_{1},\ldots,i_{k}}$
where $k$ is even and $E_{1}$ is spanned by elements $a_{i_{1},\ldots,i_{k}}$
where $k$ is odd. Moreover, it is easy to check that $E_{0}$ is the center of
$E$.

Suppose that $A$ is a $G_{2}=C_{2}\times G-$graded algebra. We denote
by $E(A)$, \emph{the Grassmann envelope of $A$, }
\[
E(A)=\bigoplus_{g\in G}\left(A_{(0,g)}\otimes E_{0}\right)\oplus\left(A_{(1,g)}\otimes E_{1}\right).
\]
which may be viewed as a $G$-graded
algebra and also as a $G_{2}-$graded algebra.

A basic property of the Grassmann envelope is the following (proof is omitted).
\begin{prop}
\label{prop:basic grassmann}$Id_{G_{2}}(A)\subseteq Id_{G_{2}}(B)\Longleftrightarrow Id_{G_{2}}(E(A))\subseteq Id_{G_{2}}(E(B)).$
\end{prop}

\section{\label{sec:Verbally-prime} $G-$verbally prime $T$-ideals}

In this section we prove Theorem \ref{classification of G-graded verbally prime}.

In order to proceed, let us recall the $G$-graded analog of Kemer's Representability theorem (see \cite{AB}). As usual, all algebras are defined over $F$, an algebraically closed field of characteristic zero.

\begin{thm*}[A]\label{Representability G-graded} The following hold.

\begin{enumerate}
\item Let $W$ be an affine $G-$graded algebra which is PI as an ungraded algebra. Then $W$ is $G-$graded PI equivalent to a finite
dimensional $G-$graded algebra $A$.

\item Let $W$
be a $G-$graded algebra which is PI as an ungraded algebra, then $W$ is $G-$graded PI equivalent
to the Grassmann envelope of an affine $G_{2}=C_{2}\times G-$graded
algebra $W_{0}$. Thus, by part $1$, $W$ is $G-$graded PI equivalent to the Grassmann
envelope of a finite dimensional $G_{2}-$graded algebra $A$.
\end{enumerate}

\end{thm*}

We start with the proof of Theorem \ref{classification of G-graded verbally prime}.

\begin{lem}
\label{product_verbally_prime} Let $A=A_{1}\times\cdots\times A_{k}$ be a direct product of $G-$graded algebras. If $A$ is verbally prime then $Id_{G}(A)=Id_{G}(A_{i_{0}})$
for some $i_{0}$. Moreover, if $A=A_{1}\times\cdots\times A_{k}$
is a product of $G_{2}-$graded algebras such that $E(A)$ is $G-$verbally
prime, then $Id_{G}(E(A))=Id_{G}(E(A_{i_{0}}))$ for some $i_{0}$.
\end{lem}

\begin{proof}
If $Id_{G}(A_{1})$ is contained in $Id_{G}(A)$, then $Id_{G}(A)=Id_{G}(A_{1})$
and we are done. Moreover, if $Id_{G}(A'=A_{2}\times\cdots\times A_{k})\subseteq Id_{G}(A)$,
then $Id_{G}(A')=Id_{G}(A)$ and so if the theorem holds for $A'$
it surely holds for $A$. Therefore, by induction we may assume that $k>1$ and $Id_{G}(A_{1}),Id_{G}(A')\nsubseteq Id_{G}(A)$ and $Id_{G}(A_{1})Id_{G}(A') \subseteq Id_{G}(A_{1})\cap Id_{G}(A')=Id_{G}(A)$. This contradicts $A$ being verbally prime, so $k=1$ as desired.

The second part follows from the first part and from the equality
\[
E(A)=E(A_{1})\times\cdots\times E(A_{k}).
\]
\end{proof}

The following definition is nonstandard. It is used only in this section.

\begin{defn}
A finite dimensional algebra $A$ over $F$ is said to be is $ [G:G]-$ \emph{minimal} (or just $G$-minimal) if it is $G$-graded and for any finite dimensional $G$-graded algebra $B$ we have

\[
Id_{G}(A)=Id_{G}(B)\Longrightarrow\mbox{dim}_{F}A\leq\mbox{dim}_{F}B.
\]

We say that a finite dimensional algebra $A$ is $[G_{2}:G]-$ \emph{minimal} if it is $G_{2}$-graded and for any finite dimensional $G_{2}$-graded algebra $B$ we have

\[
Id_{G}(E(A))=Id_{G}(E(B))\Longrightarrow\mbox{dim}_{F}A\leq\mbox{dim}_{F}B.
\]
\end{defn}

Clearly, any finite dimensional $G$-graded algebra is $G$-PI equivalent to some $[G:G]$-minimal algebra. Similarly, the Grassmann envelope $E(A)$ of any finite dimensional $G_{2}-$graded algebra is $G$-graded PI equivalent to $E(B)$ for some $[G_{2}:G]-$minimal algebra $B$. It is therefore sufficient to characterize $T$-ideals by minimal algebras.

\begin{lem}
\label{minimal_verbally_prime}The following holds.

\begin{enumerate}

\item

If $A$ is $G-$minimal and $A$
is $G-$verbally prime, then $A$ is semisimple.

\item
If $A$ is $[G_{2}:G]$-minimal and $E(A)$
is $G-$verbally prime, then $A$ is semisimple.

\end{enumerate}
\end{lem}

\begin{proof}
We show the Jacobson radical $J$ of $A$ is zero. Since $J$ is $G$-graded (resp. $G_{2}-$graded)(see \cite{Cohen-Montgomery}), the algebra $A_{ss}=A/J$ is semisimple $G$-graded (resp. $G_{2}-$graded). Suppose $J\neq 0$. Then, by minimality of $A$, we have $Id_{G}(A)\varsubsetneq Id_{G}(A_{ss})$ (resp. $Id_{G}(E(A))\varsubsetneq Id_{G}(E(A_{ss}))$). Let $p$ be any multilinear polynomial in $Id_{G}(A_{ss})\setminus Id_{G}(A)$ (resp.
$Id_{G}(E(A_{ss}))\setminus Id_{G}(E(A))$). If we consider evaluations
on the spanning set $A_{ss}\cup J$ (resp. $E(A_{ss})\cup E(J)$), we see that these evaluations are nonzero only if at least one of its variables is evaluated on an element of $J$ (resp. $E(J)$). Thus, $p(A)\subseteq J$ (resp. $p(E(A))\subseteq E(J)$). Since $J$ (resp. $E(J)$) is a nilpotent ideal, there is some integer $2 \leq m \leq$ the
nilpotency index of $J$ (resp. $E(J)$), such that $\underset{m-1}{\underbrace{p(X)p(X)\cdots p(X)}}\notin Id_{G}(A)$ (resp. $Id_{G}(E(A))$)  but $\underset{m-1}{\underbrace{p(X)p(X)\cdots p(X)}}z_{g}p(Y)\in Id_{G}(A)$ (resp. $Id_{G}(E(A))$).

The following lemma shows that $A$ (resp. $E(A)$) is not verbally prime. Contradiction.
\end{proof}

\begin{lem}
\label{not_verbally_prime}A $G-$graded algebra $A$ is \textbf{not}
verbally prime if and only if the following holds: There are two $G-$graded
polynomials $p(X),q(Y)$ defined on disjoint sets of variables which are not in $Id_{G}(A)$, but $pz_{g}q\in Id_{G}(A)=\Gamma$
for every $g\in G$. Equivalently, $A$ is $G$-graded verbally prime if and only if for any two polynomials $p$ and $q$ defined on disjoint sets of variables, nonidentities of $A$, the product $pz_{g}q$ is not a $G$-graded identity for some $g \in G$. Here $z_g$ is an extra variable, $g$-homogeneous, not in $X \cup Y$.
\end{lem}
\begin{proof}
Suppose $A$ is verbally prime and take two polynomials $p(X),q(Y)$,
with disjoint sets of variables, such that $pz_{g}q\in\Gamma$ for
every $g\in G$. We show $p$ or $q$ is in $\Gamma$. Since every element of $\left\langle p\right\rangle _{T}\left\langle q\right\rangle _{T}$
is $T-$generated by $pz_{g}q$, where $g$ runs over the elements
of $G$, we get that $\left\langle p\right\rangle _{T}\left\langle q\right\rangle _{T}\subseteq\Gamma$ and the result follows by definition of verbally prime.

Now suppose that $A$ is not verbally prime and let $I,J$ be two $G-$graded
$T-$ideals not contained in $\Gamma$ such that $IJ\subseteq\Gamma$.
Choose $p\in I-\Gamma$ and $q\in J-\Gamma$. Since $pz_{g}\in I$, for every $g\in G$, we have (by our assumption) that
$pz_{g}q\in\Gamma$. But clearly, we may assume
that $p$ is defined on a set of variables disjoint from the ones of
$q$ and $z_{g}$ and so we are done.
\end{proof}

We can now prove Theorem \ref{classification of G-graded verbally prime}.

\begin{proof}

Suppose $\Gamma$ contains a Capelli polynomial (or equivalently, $\Gamma$ is the $T$-ideal of an affine $G$-graded algebra which is ungraded PI (see\cite{AB})). Then by part $1$ of Theorem A, $\Gamma=Id_{G}(A)$, where $A$ is finite dimensional $G-$graded algebra.
Assume also that $A$ is $G-$minimal. By Lemma \ref{minimal_verbally_prime}
we obtain that $A$ is semisimple. Since every semisimple $G-$graded
algebra is a direct product of $G-$simple algebras, invoking Lemma \ref{product_verbally_prime}
we can replace $A$ by a $G-$simple algebra. In the general case, by part 2 of Theorem A, $\Gamma=Id_{G}(E(A))$, where $A$ is finite dimensional $G_{2}-$graded algebra. We assume that $A$ is $[G_{2}:G]-$minimal and continue the proof as above.
This proves one direction of each one of the statements in the theorem.

For the converse, suppose $A$ is a finite dimensional $G-$graded simple algebra.
Since the evaluation of any $T-$ideal $I$ on $A$ is a $G-$graded
ideal $\tilde{I}$ of $A$, in case $I\nsubseteq\Gamma$
we get $\tilde{I}=A$. So if $I,J\nsubseteq\Gamma$ we have $\widetilde{IJ}=\tilde{I}\tilde{J}=A$,
so $IJ\nsubseteq\Gamma$. Thus, $A$ is verbally prime.

Finally, we prove that if $A$
is a $G_{2}-$simple finite dimensional algebra, then $E(A)$ is $G-$verbally prime. Suppose $I$ is a
$G-$graded $T-$ideal not contained in $\Gamma=Id_{G}(E(A))$. Denote

\[
I_{A}=\left\{ a\in A|b\otimes a=f(\tilde{x}_{1},\ldots,\tilde{x}_{n}),\, b\neq0,f\in I \subseteq  F\langle X_{G} \rangle \right\}
\]
Here, $f$ is multilinear and $G$-homogeneous in $F\langle X_{G} \rangle$ and $\tilde{x}_{i} \in E(A)$ is of the form $e_{i_1}e_{i_2}\cdots e_{i_k} \otimes a_i$.

Note that $Span_{F}(I_{A})$ is a nonzero $G_{2}-$graded ideal of $A$. Thus $Span_{F}(I_{A})=A$. Hence, for any two $G-$graded $T-$ideals $I,J$ not contained in $\Gamma$, we have $Span_{F}(I_{A})Span_{F}(J_{A})=A$ and hence $I(A)J(A) \neq 0$ which says that $IJ$ is not contained in $\Gamma$. This completes the classification of $G-$graded verbally prime ideals.

\end{proof}

\section{\label{sec:strongly-verbally-prime}strongly verbally prime ideals}

As mentioned in the introduction, we shall give a complete characterization of $G$-graded strongly verbally prime ideals $\Gamma$ in case $\Gamma$ is the $T$-ideal of an affine $G$-graded algebra which is ungraded PI.

The first step is to show

\begin{cor}
\label{verbally_vs_strongly}If $\Gamma$ is $G$-graded strongly verbally
prime, then it is $G$-graded verbally prime.\end{cor}
\begin{proof}
Suppose not and consider $p,q$ as in the second part of the proof of  Lemma \ref{not_verbally_prime}.
Since $q\notin\Gamma$, $pz_{g}\in\Gamma$ for every $g\in G$. Similarly,
since $p\notin\Gamma$, $z_{g}\in\Gamma$ for all $g\in G$. We get
that $\Gamma$ is equal to the whole free algebra and
so $p$ (and $q$) are in $\Gamma$. Contradiction.
\end{proof}
In case $G$ is trivial it is not hard to prove (see \cite{GiaZai}, Lemma
1.3.10) that a vector space $T-$generated
by a single polynomial $p$ is in fact a Lie $T-$ideal, so any
one-sided $T-$ideal must be a two-sided $T-$ideal. Hence, $pz_{e}$
is in the left $T-$ideal $\left\langle p\right\rangle_{T}$, so $pz_{e}q\in\left\langle pq\right\rangle _{T}$.
In other words $pz_{e}q$ is an identity if $pq$ is such. Thus,
in the ungraded case the opposite of Corollary \ref{verbally_vs_strongly}
holds.

In contrast to the ungraded case, we shall see that if the group $G$ is nontrivial then there exist algebras
which are verbally prime but not strongly verbally prime. For instance, it follows easily from our classification that the algebra of $3 \times 3-$matrices over $F$ with the elementary $\mathbb{Z}_{2}-$grading given by the vector $(e,e,\sigma)$ where $\mathbb{Z}_{2} = \langle \sigma \rangle$ is verbally prime but not strongly verbally prime.

Before embarking into the proof of Theorem \ref{main theorem1} we need some more preparations. Let $\Gamma$ be a $G$-graded strongly verbally prime $T$-ideal which contains a Capelli polynomial.
By Kemer's theory for $G$-graded algebras (see \cite{AB}) we have that $\Gamma$ is the $T$-ideal of identities of a $G$-graded finite dimensional algebra $A$ over $F$. Furthermore, by Theorem \ref{classification of G-graded verbally prime} and the fact that strongly verbally primeness implies verbally primeness (Corollary \ref{verbally_vs_strongly}) we may assume the algebra $A$ is $G$-simple. Our goal in the first part of the proof is to find necessary conditions on presentations of \textit{the} $G$-simple algebra $A$ so that $Id_{G}(A)$ is $G$-graded strongly verbally prime.

The word \textit{the} in the last sentence needs justification. It is trivial to construct nonisomorphic finite dimensional algebras which are PI equivalent. On the other hand, PI equivalent finite dimensional simple algebras over an algebraically closed field $F$ are $F-$isomorphic by e.g. the Amitsur-Levitzki theorem. This is true also for finite dimensional $G$-simple algebras over $F$. Indeed, it was proved by Koshlukov and Zaicev that if finite dimensional $G$-simple algebras $A$ and $B$, $G$ abelian, are $G$-graded PI equivalent, then they are $G$-graded isomorphic \cite{Koshlukov-Zaicev}. This result was extended to arbitrary groups by the first named author of this article and Haile \cite{AljHaile}. In particular, note that if $A$ is $G$ -graded verbally prime, the minimal algebra $B$ with $Id_{G}(B)=Id_{G}(A)$ is uniquely determined up to a $G$-graded isomorphism. Indeed, as noted in the proof of Theorem \ref{classification of G-graded verbally prime} any $G$-minimal algebra and $G$-verbally prime is necessary $G$-simple. We record this in the following

\begin{prop} $G$-minimal, $G$-graded verbally prime algebras $B$ and $B'$ are $G$-graded PI equivalent if and only if they are $G$-graded isomorphic.
\end{prop}
As mentioned in the introduction, the $G$-graded structure of a finite dimensional $G$-simple algebra is presented by a triple $(H,[c],\mathfrak{g})$.
We recall that a $G$-graded algebra may admit different presentations, that is algebras with presentations $(H,[c],\mathfrak{g})$ and
$(H',[c'],\mathfrak{g}')$ (different parameters) may be $G$-graded isomorphic. In fact, in \cite{AljHaile} the authors
present three type of \textit{moves} on presentations, $M_1, M_2, M_3$ (see below) such that

\begin{enumerate}
\item

 $M_{i}((H,[c],\mathfrak{g}))\underset{G}{\backsimeq}(H,[c],\mathfrak{g})$, $i=1,2,3$
($G-$isomorphic)

\item

$(H,[c],\mathfrak{g}))\underset{G}{\backsimeq}(H',[c'],\mathfrak{g}')$ if and only if $(H',[c'],\mathfrak{g}')$ may be obtained from $(H,[c],\mathfrak{g})$ by composing a sequence of such moves.

\end{enumerate}

For the reader convenience let us recall the moves $M_1, M_2, M_3$.

$M_1:$ $\mathfrak{g}'=(g_{\sigma(1)},\ldots,g_{\sigma(m)})$, where $\sigma\in S_{m}$,
and $H=H',c=c'$.

$M_2:$ $\mathfrak{g}'=(h_{1}g_{1},\ldots,h_{m}g_{m})$, where $h_{1},\ldots,h_{m}\in H$,
and $H=H',c=c'$.

$M_3:$ For $g\in G$. $\mathfrak{g}'=(gg_{1},\ldots,gg_{m}),\, H'=gHg^{-1}$
and $c'\in Z^{2}(H',F^{*})$ is defined by $c'(gh_{1}g^{-1},gh_{2}g^{-1})=c(h_{1},h_{2})$.

Applying these moves we obtain easily the following normalization on the presentation of a $G$-simple algebra (see also \cite{AljHaile}).

\begin{cor}
Any finite dimensional $G$-simple algebra has a presentation with
$\mathfrak{g}=(\underset{m_{1}}{\underbrace{g_{1},\ldots,g_{1}}},\ldots,\underset{m_{k}}{\underbrace{g_{k},\ldots,g_{k}}})\in G^{m}$,
where the $g_{i}$'s are different representatives of right $H-$cosets,
$0 < m_{1}\leq\cdots\leq m_{k}$ and $g_{1}=e$.

\end{cor}

Let $P_{A}= (H,[c],\mathfrak{g})$ be a presentation of a $G$-simple algebra $A$. Suppose the presentation is normalized as above.

\bigskip

We denote by $e_{r,s}^{(i,j)}$ the element
\[
e_{m_{0}+\cdots m_{i-1}+r,m_{0}+\cdots m_{j-1}+s}\in M_{m}(F)
\]
where $1\leq i,j \leq k$, $m_{0}=0$ and $r\in\{1,\ldots,m_{i}\}$, $s\in\{1,\ldots,m_{j})$.
In case $i=j$ we slightly simplify the notation and replace $e_{r,s}^{(i,i)}$ by $ e_{r,s}^{(i)}$.

For a fixed $i$ we refer to the linear span of $\{u_{e}\otimes e_{r,s}^{(i)}\}_{r,s}$
as the $i$th $e-$block of $A$. Note that the linear span
of all $e-$blocks $(i=1,\ldots,k)$ is equal to $A_{e}$, the $e-$component
of $A$.

Suppose $A$ is $G$-graded strongly verbally prime. Notation as above.

\begin{lem}
\label{lem:-same_block_size}The following hold.
\begin{enumerate}

\item
$k=[G:H]$

\item
$m_{1}=m_{k}$ and so $m_{1}=\cdots=m_{k}$

\end{enumerate}

\end{lem}
\begin{proof}

Let us show first all right $H$-cosets are represented in $\mathfrak{g}$. If not, it follows from Lemma $3.3$ in \cite{AljDavid} that the $G$-grading on $A$ is \textit{degenerate} (recall that a $G$-grading on an algebra $A$ is nondegenerate if the product of homogeneous components $A_{g_1}A_{g_2}\cdots A_{g_s}$ is nonzero, for any $s > 0$ and any $s$-tuple $(g_1.\ldots,g_s) \in G^{(s)}$; equivalently, the $T$-ideal of $G$-graded identities has no $G$-graded multilinear monomials). It follows that $Id_{G}(A)$ contains multilinear monomials. Since the grading is assumed to be connected, we have monomials of the form $x_{g_1}x_{g_2}\cdots x_{g_n} \in Id_{G}(A)$ where $g_i \in Supp_{G}(A)$ and hence $x_{g_i} \notin Id_{G}(A)$. Clearly, this shows that $A$ is not $G$-graded strongly verbally prime and the first part of the lemma is proved.

Next, suppose by contradiction that $m_{1}<m_{k}$ ($m_{i}$ are positive integers).

For $i=1,\dots,k$ consider the following monomial of $e$-variables

$$Z_{i}=y_{0}^{(i)}x_{1}^{(i)}y_{1}^{(i)}\cdots x_{m_{i}^{2}}^{(i)}y_{m_{i}^{2}}^{(i)}$$
and denote
$$Z=Z_{1}w_{1}Z_{2}\cdots w_{k-1}Z_{k},$$ where $w_{i}$
is a $g_{i}^{-1}g_{i+1}-$graded variable. Finally let $p=\mbox{Alt}_{X}Z$,
where $X$ is the set of all $x$ variables. We refer to the
$y$ variables as \emph{frames} and the $w$ variables as \emph{bridges}.

Let us show that $p$ is a nonidentity of $A$. It is well known that
one can find an ordering on the elementary matrices
$e_{r_{1},s_{1}},\ldots,e_{r_{m_{i}^{2}},s_{m_{i}^{2}}}$ of $M_{m_{i}}(F)$ such that
\begin{enumerate}
\item $r_{1}=s_{m_{i}^{2}}=1$.
\item $s_{t}=r_{t+1}$ for $t=1,\ldots,m_{i}^{2}-1$.
\end{enumerate}

Evaluate $x_{t}^{(i)}:=\tilde{x}_{t}^{(i)}=u_{e}\otimes e_{r_{t},s_{t}}^{(i)}$.
Notice that there is a unique evaluation of the $y_{t}^{(i)}$'s by $e$-homogeneous elements such that $\tilde{Z_i}\neq 0$, namely
$\tilde{y}_{t}^{(i)}=u_{e}\otimes e_{s_{t},s_{t}}^{(i)}$ if $t>0$ and $\tilde{y}_{0}^{(i)}= \tilde{x}_{1}^{(i)}$ if $t=0$. This yields
$\tilde{Z}_{i}=u_{e}\otimes e_{1,1}^{(i)}$. Finally, we set $\tilde{w}_{i}=u_{e}\otimes e_{1,1}^{(i,i+1)}$ and obtain $\tilde{Z}=u_{e}\otimes e_{1,1}^{(1,k)}$. We claim that any
nontrivial alternation of the $\tilde{x}$'s gives zero. Indeed,
if the alternation sends an $\tilde{x}$ corresponding to $Z_{i}$
to $\tilde{x}'$ corresponding to $Z_{j}$ ($i\neq j$), we get
zero since the frames of $Z_{j}$ are evaluated on elements of the
$j-$block and these do not match with the indices of $\tilde{x}$. On the other hand if the alternation
sends some $\tilde{x}$ of $Z_{i}$ to a different $\tilde{x}'$ of $Z_{i}$
we also get zero since $x$ and $x'$ are evaluated on different
elementary matrices and again the indices of the frames do not match. All in all, we get that
$\tilde{p}=\tilde{Z}\neq0$.

Next, we claim that $p(A)\subseteq F^{c}H\otimes R$, where $R$ is
spanned by $e_{r,s}^{(i,j)}$, where
$m_{i}=m_{1}$ and $m_{j}=m_{k}$ (in particular $i < j$).
Indeed, since $p$ is multilinear it is enough to consider only evaluations
on elements of the form $u_{h}\otimes e_{i,j}$. Consider such an
evaluation: $\tilde{x},\tilde{y},\tilde{w}$. Obviously, we may assume that $\tilde{p}\neq0$
and so without loss of generality we may assume that $\tilde{Z}\neq0$.

Note that since the $\tilde{x},\tilde{y}$ (alternating variables and frame variables) have been chosen to be homogeneous of degree $e$, all variables of each $Z_{i}$ must be evaluated on elements of the same $e-$block. Moreover, since the $x$ variables are alternating
in $p$, their evaluation must consist of linearly independent elements, thus
consists of a basis of $A_{e}$. Thus, the $x$ variables of each $Z_{i}$
are evaluated on a basis of some $e-$block of $A_{e}$ where for $i \neq j$ the evaluation is on different $e$-blocks.
It follows that each $Z_{i}$ must be evaluated on elements of an $e-$block
of dimension $m_{i}^{2}$. Consequently $Z(A)\subseteq F^{c}H\otimes R$ (evaluations of the monomial $Z$ on $A$). In order to show that $p(A)\subseteq F^{c}H\otimes R$ we proceed as follows.
Any evaluated monomial $\widetilde{Z'}$ resulting from a permutation
of the evaluated $x$ variables of $Z$ is in fact equal to
$\tilde{Z}$ after reevaluation of the $x$ variables. The previous argument
shows that $Z'(A)\subseteq F^{c}H\otimes R$ and hence also $p(A)\subseteq F^{c}H\otimes R$.

\begin{rem}
In fact we can strengthen the statement in the last part of the proof above. We claim that if the following conditions hold

\begin{enumerate}

\item

All variables are evaluated on elements of $A$ of the form $u_{h}\otimes e_{i,j}$

\item

The variables of $Z_{i}$ are homogeneous of degree $e$.

\item

$Z(A)\neq 0$
\end{enumerate}
then for any monomial $Z'$ in $p$ obtained from $Z$ by a nontrivial permutation of the $x$-variables we have $\widetilde{Z'}=0.$
Indeed, as noted above, for any nonzero evaluation of the polynomial $p$ (i.e. an evaluation of the $x$'s, the $y$'s and the $w$'s) which yields a nonzero value of the monomial $Z$, we have that the value of the $x$'s is determined by the value of the $y$'s (frames). Since $Z'$ is obtained from $Z$ by permuting nontrivially the $x$'s and keeping the $y$'s, its value $\widetilde{Z'}$ must be zero.

\end{rem}

We complete now the proof of the lemma. Since $m_1 < m_k$, we have $\left(F^{c}H\otimes R\right)^{2}=0$. It follows that $pq\in Id_{G}(A)$
where $q$ is the polynomial $p$ but with a disjoint set of variables
- contradiction to $A$ being strongly verbally prime.

\end{proof}

We let $\mathfrak{g}_{\mathbf{set}}=\{g_{1},\ldots,g_{k}\}$ be a set of representatives of the right cosets of $H$ in $G$. Denote $m_{\mathbf{block}}=m_{1}=\cdots=m_{k}$.
\begin{prop}\label{normal subgroup}
$\mathfrak{g}_{\mathbf{set}}=\{g_{1},\ldots,g_{k}\}$ normalizes $H$ $($i.e. $H$ is normal in $G$$)$.

\end{prop}

\begin{proof}
First let us define an equivalence relation $\sim$ on $\mathfrak{g}_{\mathbf{set}}$:
\[
g_{i}\sim g_{j}\mbox{ iff }g_{i}^{-1}Hg_{j}\cap H\neq\emptyset,
\]
and denote by $[g]$ the equivalence class of $g\in\mathfrak{g}_{\mathbf{set}}$. Clearly, the equivalence relation among the $g_{i}$'s determines an equivalence relation among the corresponding $e$-blocks.
Notice that the claim of the proposition holds if and only if the equivalence
classes are of cardinality $1$.

Let $\mathfrak{g}_{\sim}=\{\gamma_{1},\ldots,\gamma_{\tilde{k}}\} \subseteq \mathfrak{g}_{\mathbf{set}}$, $\tilde{k} \leq k$, be
a full set of representatives of $\sim$. Changing representatives we may write

\[
(\underset{m_{\mathbf{block}}}{\underbrace{[g_{1}],\ldots,[g_{1}]}},\ldots,[\underset{m_{\mathbf{block}}}{\underbrace{g_{k}],\ldots,[g_{k}]}})=(\underset{\tilde{m}_{1}}{\underbrace{[\gamma_{1}],\ldots,[\gamma_{1}]}},\ldots,[\underset{\tilde{m}_{\tilde{k}}}{\underbrace{\gamma_{\tilde{k}}],\ldots,[\gamma_{\tilde{k}}]}}),
\]
where $\tilde{m}_{i}=m_{\mathbf{block}}\cdot|[\gamma_{i}]|$. Here $|[\gamma_{i}]|$ denotes the cardinality of the equivalence class in $\mathfrak{g}_{\mathbf{set}}$ represented by $\gamma_{i}$. Moreover, since $e \in G$
normalizes $H$, we may assume

\begin{enumerate}
\item $[\gamma_{1}]=[e](=\{e\})$.
\item $0 < \tilde{m}_{1}\leq\cdots\leq\tilde{m}_{\tilde{k}}$.
\end{enumerate}

We make similar conventions to the ones we made before Lemma \Lemref{-same_block_size}.
Denote by $e_{r,s}^{[i,j]}$ the element $e_{\tilde{m}_{0}+\cdots\tilde{m}_{i-1}+r,\tilde{m}_{0}+\cdots\tilde{m}_{j-1}+s}\in M_{m}(F)$,
where $\tilde{m}_{0}=0$ and $r\in\{1,\ldots,\tilde{m}_{i}\}$, $s\in\{1,\ldots,\tilde{m}_{j}\}, 1\leq i,j \leq \tilde{k}$.
We say that $u_{h}\otimes e_{r,s}^{[i,j]}$ is in the $[i,j]-$block.
Moreover, we denote by $u_{h}\otimes e_{r,s}^{[i]}$ the element $u_{h}\otimes e_{r,s}^{[i,i]}=u_{h}\otimes e_{r,s}^{[\gamma_{i},\gamma_{i}]}$ and call its block the $[i]-$block.

The idea of the proof is somewhat similar to the one above. We construct a ``Capelli type'' polynomial $p(X)$, obtained by alternating some variables of a suitable $G$-graded monomial $Z$.

For $i=1,\dots,k$, consider the monomial of $e$-variables

$$Z_{i}=y_{0}^{(i)}x_{1}^{(i)}y_{1}^{(i)}\cdots x_{m_{i}^{2}}^{(i)}y_{m_{i}^{2}}^{(i)}$$ where $m_i=m_{\mathbf{block}}$
and let
$$Z=Z_{1}w_{1}Z_{2}\cdots w_{k-1}Z_{k},$$ where $w_{j}$
is a $g_{j}^{-1}h_{j}g_{j+1}-$graded variable and $h_{j} \in H$ to be chosen, $j=1,\ldots,k-1$. If $g_{j}$ and $g_{j+1}$ are nonequivalent we know that $g_{j}^{-1}h_{j}g_{j+1}$ is not in $H$ for any $h_{j}\in H$ and so we take $h_{j}=e$ for simplicity. On the other hand if $g_{j}$ and $g_{j+1}$ are equivalent we choose $h_{j}=\hat{h}_{j} \in H$ such that $g_{j}^{-1}\hat{h}_{j}g_{j+1}\in H$. Note that by our construction $w_1, w_2, \ldots, w_{\tilde{m}_{1}/m_{\mathbf{block}}-1}$ are elements in $H$, $w_{\tilde{m}_{1}/m_{\mathbf{block}}} \notin H$, then the next ${\tilde{m}_{2}/m_{\mathbf{block}}-1}$ variables $w$'s  have homogeneous degrees in $H$, followed by $w_{\tilde{m}_{1}+\tilde{m}_{2}/m_{\mathbf{block}}}$ whose homogeneous degree is not in $H$ and so on.

Let us introduce an equivalence relation on submonomials: we say that submonomials $Z_{i}$ and $Z_{j}$ are equivalent if and only if $i=j$ or they are separated by $w$'s whose homogeneous degrees is in $H$. We denote the equivalent class of $Z_{i}$ by $[Z_{i}]$. Note that we have constructed precisely $\tilde{k}$ equivalence classes. Let $p(X)$ be the polynomial obtained by alternating all $x_r$'s variables of $Z$.

\begin{lem}
The polynomial $p(X)$ is a graded nonidentity of $A$. Furthermore, evaluating the polynomial $p$ on $A$ yields elements in $F^{c}H\otimes R$, where $R$ is spanned by elements $e_{r,s}^{[i,j]}$ with $i$ and $j$ satisfying $\tilde{m}_i = \tilde{m}_{1}$ and $\tilde{m}_{j} = \tilde{m}_{\tilde{k}}$.

\end{lem}
Using the lemma, since $[F^{c}H\otimes R]^{2}=0$, by taking the product of two such polynomials as $p(X)$ with disjoint sets of variables, we see that $A$ is not strongly verbally prime and so the proof of the proposition is complete. Note that with this we have proved the necessity of conditions $1$ and $2$ in Theorem \ref{main theorem1}.

\begin{proof} (of the lemma)
By the construction the monomial $Z$ has a nonzero value on $A$ as follows. We evaluate the $x$'s in each segment $Z_{i}$ of the monomial $Z$ on $e$-elements of the form $u_{e}\otimes e_{r,s}^{[i]}$, the $y$'s (frames) on $e$-elements of the form $u_{e}\otimes e_{r,r}^{[i]}$ and finally the variables $w_{i}$ (bridges), $i=1,\ldots,k-1$, on $u_{h_i}\otimes e_{1,1}^{[i,i+1]}$ where $h_{i}=e$ if $g_{i}$ and $g_{i+1}$ are nonequivalent and $h_{i}= \hat{h}_{i}$ if $g_{i}$ and $g_{i+1}$ are equivalent. Other monomials of $p(X)$ obtained from $Z$ by alternating the $x$'s yield a zero value since the values on the $y$'s and the $x$'s do not match. This shows $p(X)$ is a nonidentity of $A$.

In order to prove the second claim of the lemma, note that by the multilinearity of $p(X)$ it is sufficient to check evaluations on basis elements of the form $u_{h}\otimes e_{r,s}^{[i,j]}$. From now on we fix an evaluation of $p(X)$ of that kind. We'll show the value is in $F^{c}H\otimes R$. Clearly, we may assume the value is nonzero and so there exists a monomial $Z$ whose value is nonzero. In that case, by the alternation of the $x$'s in $p(X)$, their values must consist of the full basis of the $e$-homogeneous component of $A$. Since the $y$'s are homogeneous of degree $e$, their value is determined uniquely by the $x$'s and the nonvanishing of the monomial $Z$. Therefore as above, any monomial $Z'$ obtained from $Z$ by a nontrivial permutation vanishes. In order to complete the proof we must show the value of the monomial $Z$ is in $F^{c}H\otimes R$. As noted above, the evaluation of the monomial $Z$ establishes a one to one correspondence between the submonomials $Z_{i}$ and $e$-blocks. Furthermore, equivalent classes of submonomials $[Z_{i}]$ must be evaluated on equivalent classes of $e$-blocks such that the sizes of the equivalence classes coincide. The proof of the lemma now follows easily and so the proof of Proposition \ref{normal subgroup} is now complete.

\end{proof}

\end{proof}

We pause for a moment and summarize what we have so far. We showed that if $\Gamma$ is a $G$-graded $T$-ideal which contains a Capelli set and is $G$-graded strongly verbally prime then it is the $T$-ideal of identities of a finite dimensional $G$-simple algebra. Of course, there may be other algebras (finite dimensional or not) whose $T$-ideal of $G$-graded identities is $\Gamma$. However, if we restrict our attention to finite dimensional $G$-simple algebras $A$ whose $Id_{G}(A)=\Gamma$, the algebra $A$ is uniquely determined up to a $G$-graded isomorphism. We know by Bahturin-Sehgal-Zaicev's theorem that any $G$-simple algebra and hence the algebra $A$, is represented by a triple $(H,[c],\mathfrak{g})$, where $H$ is a subgroup of $G$, $[c] \in H^{2}(H, F^{*})$ and $\mathfrak{g} =(g_{1},\ldots,g_{m})\in G^{m}$. We know the presentation $(H,[c],\mathfrak{g})$ of $A$ is unique up to $3$ type of ``moves'' $M_1, M_2, M_3$. We showed that the strongly verbally prime condition on $A$ implies the group $H$ is normal in $G$ and $H$ cosets of $G$ are represented equally in  $\mathfrak{g} =(g_{1},\ldots,g_{m})\in G^{m}$. In particular $m/[G:H]$ is an integer.

The above considerations lead to the following characterization of $G$-simple crossed products. We recall

\begin{defn}
A $G$-graded algebra $B$ is said to be a $G$-crossed product if $B_{g}$ has an invertible element for any $g \in G$.
\end{defn}
Note that this implies $dim_{F}(B_{g}) = dim_{F}(B_{g'})$ for any $g, g' \in G$.

\begin{prop}

Let $A$ be a $G-$simple algebra over an algebraically closed field of characteristic zero. Then $A$ is a $G$-crossed product if and only if $A$ has a presentation $(H,[c],\mathfrak{g})$ where $H$ is a subgroup of $G$ $($not necessarily normal$)$ and the right $H$ cosets of $G$ are equally represented in  $\mathfrak{g} =(g_{1},\ldots,g_{m})\in G^{m}$.
\end{prop}

\begin{proof}
Suppose $A$ has a presentation as above where right $H$-cosets are equally represented. We know the elements of the form $u_{h} \otimes e_{i,j} \otimes E$, where $i,j \in \{1,\ldots,[G:H]\}$, $E \in M_{r}(F)$ are homogeneous of degree $g_{i}^{-1}h g_{j}$ and span $A$. We note the homogeneous components are concentrated in \textit{block permutation} matrices which means that for any $g \in G$ and any $i \in \{1,\ldots,[G:H]\}$ there exist a unique $h=h(i)$ and unique $j=j(i)$ such that for every $E \in M_{r}(F)$ the element $u_{h} \otimes e_{i,j} \otimes E$ has homogeneous degree $g$. It follows that the element

$$\sum_{i\in \{1,\ldots,[G:H]\}} u_{h(i)}\otimes e_{i,j(i)}\otimes I_{r},$$ where $I_{r}$ is the identity matrix of rank $r$,
is invertible and homogeneous of degree $g$. The inverse is $$\sum_{i\in \{1,\ldots,[G:H]\}} c^{-1}(h(i),h(i)^{-1})u_{h(i)^{-1}}\otimes e_{j(i),i}\otimes I_{r}.$$
As noted above, in this case the dimensions of the homogeneous components of $A$ coincide and equal $[G:H]r^{2} = dim_{F}(A)/\mid G \mid
$.

For the converse let $r_{1},\ldots,r_{[G:H]}$ be the multiplicities of the right $H$-cosets in $\mathfrak{g}$ and suppose are not all equal. Clearly $r_{1}+\cdots + r_{[G:H]} = m$. Moreover, by definition of the $G$-grading on $A$, $dim_{F}(A_{e}) = r_{1}^{2} + r_{2}^{2} + \cdots + r_{[G:H]}^{2}$ which strictly exceeds $dim_{F}(A)/\mid G \mid
$. This implies the dimension of the homogeneous components are not equal and hence $A$ is not a $G$-crossed product.
\end{proof}

In particular we have

\begin{cor}
Let $A$ be a finite dimensional $G$-simple algebra with presentation $(H,[c],\mathfrak{g})$. Suppose $H$ is normal and multiplicities of $H$-coset representatives are equal $(=r)$. Then the homogeneous components consist of block permutation matrices.
\end{cor}

Our next step is to put restrictions (i.e. necessary conditions) on the cohomology class $[c]$ provided the $G$-simple algebra $A$ is strongly verbally prime. This will lead to a necessary condition on $[c]$ which together with the conditions on $H$ and on $\mathfrak{g}$ mentioned above turns out to be also sufficient for strongly verbally primeness. In order to proceed we recall briefly the theory and general terminology which relate the cohomology group $H^{2}(H,F^*)$ with $H$-graded PI theory (see \cite{AHN}, \cite{AljHaile}).

Let $H$ be a finite, $c: H \times H \rightarrow F^{*}$ a $2$-cocycle and $[c]$ the corresponding cohomology class. Let $F^{c}H$ be the corresponding twisted group algebra and let $\{u_{h}\}_{h \in H}$ be a set of representatives for the $H$-homogeneous components in $F^{c}H$. We have $u_{h_1}u_{h_2}= c(h_1,h_2)u_{h_1h_2}$ for $h_i \in H$. Let $n$ be a positive integer. For $\mathfrak{\mathfrak{h}}=(h_{1},\ldots,h_{n})\in H^{n}$ and $\sigma\in S_{n}$ consider

\[
Z_{\mathfrak{h},\sigma}=\prod_{i=1}^{n}x_{i.h_{i}}-\underset{\alpha_{\mathfrak{h},\sigma}}{\underbrace{\frac{c(\mathfrak{\mathfrak{h}})}{c(\mathfrak{\mathfrak{h}}_{\sigma})}}}\prod_{i=1}^{n}x_{\sigma(i),h_{\sigma(i)}},
\]
where $c(\mathfrak{\mathfrak{h}})\in F^{*}$ (slightly abusing notation) is the element determined
by $u_{h_{1}}\cdots u_{h_{n}}=c(\mathfrak{\mathfrak{h}})u_{h_{1}\cdots h_{n}}$ and $\mathfrak{\mathfrak{h}}_{\sigma}=(h_{\sigma(1)},\ldots,h_{\sigma(n)})$.
We are interested in the binomials $Z_{\mathfrak{h},\sigma}$ which
are $H$-graded identities of $F^{c}H$. This happens precisely when
\[
h_{1}\cdots h_{n}=h_{\sigma(1)}\cdots h_{\sigma(n)}.
\]
Note that if $h_{1}\cdots h_{n}=h_{\sigma(1)}\cdots h_{\sigma(n)}$, the polynomial $\prod_{i=1}^{n}x_{i.h_{i}}-\lambda \cdot\prod_{i=1}^{n}x_{\sigma(i),h_{\sigma(i)}}$ is an $H$-graded identity if and only if $\lambda = \alpha_{\mathfrak{h},\sigma}$.

We refer to these binomials as the \emph{binomial identities} of the twisted group algebra $F^{c}H$. It is known that the binomial identities generate the $T$-ideal of $H$-graded identities of $F^{c}H$ (see \cite{AljHaile} ,Theorem $4$). Furthermore, one sees easily that the scalars $\alpha_{\mathfrak{h},\sigma}$ depend on the cohomology class $[c]$ but not on the representative $c$.

Next it is convenient to introduce into our discussion presentations of the group $H$.
Let $\mathcal{F}$ be the free
group generated by $\Omega=\left\{ x_{i,h}|\, i\in\mathbb{N}\,,h\in H\right\} $
and consider the exact sequence
\[
\xymatrix{1\ar[r] & \mathfrak{R}\ar[r] & \mathcal{F}\ar[r] & H\ar[r] & 1}
\]
where the map $\mathcal{F}\to H$ is given by $x_{i,h}\mapsto h$.
This induces the central extension
\[
\xymatrix{1\ar[r] & \mathfrak{R}/[\mathcal{F},\mathfrak{R}]\ar[r] & \mathcal{F}/[\mathcal{F},\mathfrak{R}]\ar[r] & H\ar[r] & 1}
.
\]
Consider the central extension
\[
\xymatrix{1\ar[r] & F^{*}\ar[r] & \Gamma\ar[r] & H\ar[r] & 1}
\]
which corresponds to the cocycle $c$ (i.e. $\{u_{h}|\, h\in H\}$
is a set of representatives in $\Gamma$ and $u_{h_{1}}u_{h_{2}}=c(h_{1},h_{2})u_{h_{1}h_{2}}$).
The map $\psi_{c}:\mathcal{F}/[\mathcal{F},\mathfrak{R}]\rightarrow\Gamma$,
given by $\psi_{c}(x_{i,h})=u_{h}$, induces a map of extensions and
its restriction $\psi_{c}:[\mathcal{F},\mathcal{F}]\cap\mathfrak{R}/[\mathcal{F},\mathfrak{R}]=M(H)\to F^{*}$
is independent of the presentation and corresponds to $[c]$.
That is, $\Phi([c])=\psi_{c}$ where $\Phi: H^{2}(G,F^{*}) \rightarrow Hom_{\mathbb{Z}}(M(H), F^{*})$ is the well known isomorphism of the Universal Coefficient Theorem. Here, $M(H)$ is the Schur multiplier of the group $H$ and the equality with $[\mathcal{F},\mathcal{F}]\cap\mathfrak{R}/[\mathcal{F},\mathfrak{R}]$ is the well known Hopf formula.
It follows that
\begin{multline*}
\Phi([c])\left(\underset{z}{\underbrace{x_{1,h_{1}}\cdots x_{n,h_{n}}(x_{\sigma(1),h_{\sigma(1)}}\cdots x_{\sigma(n),h_{\sigma(n)}})^{-1}}}\right)=u_{h_{1}}\cdots u_{h_{n}}(u_{h_{\sigma(1)}}\cdots u_{h_{\sigma(n)}})^{-1}\\
=\frac{c(\mathfrak{\mathfrak{h}})}{c(\mathfrak{\mathfrak{h}}_{\sigma})}=\alpha_{\mathfrak{h},\sigma}
\end{multline*}
Now, it is clear that $z\in[\mathcal{F},\mathcal{F}]$ and moreover, by the condition we imposed on $\mathfrak{\mathfrak{h}}$ and $\sigma \in S_{n}$,
$z\in\mathfrak{R}$ and hence $z \in \mathfrak{R} \cap [\mathcal{F},\mathcal{F}]$. It is not difficult to show that modulo
$[\mathcal{F},\mathfrak{R}]$ all elements of $M(H)$ are of this
form (see \cite{AHN}, Lemma $11$), and so, we readily see that the binomial identities of the twisted group algebra $F^{c}H$ determine uniquely the cohomology class $[c]$.

We are ready to prove the necessity of the third condition in Theorem \ref{main theorem1} for $A$ to be
strongly verbally prime, namely the cohomology class $[c]$ is $G$-invariant.

Recall that $G$ acts on $H^{2}(H,F^{*})$ via conjugation, that is for $[\beta]\in H^{2}(H,F^{*})$

\[
\left(g\cdot\beta\right)(h_{1},h_{2})=\beta(gh_{1}g^{-1},gh_{2}g^{-1}).
\]
Applying the isomorphism $\Phi$ one checks the cohomology class $[c]$ is $G-$invariant if and only if
\begin{equation}
\forall g\in G:\,\,\alpha_{\mathfrak{h},\sigma}=\alpha_{g\mathfrak{h}g^{-1},\sigma},\label{eq:0}
\end{equation}
for all $\alpha_{\mathfrak{h},\sigma}$ corresponding to the binomial identities $\{Z_{\mathfrak{h},\sigma}\}$.

\begin{lem}\label{invariance of cohomology}
\label{lem:-is-invariant}Suppose $A$ is strongly verbally prime with presentation $(H,[c],\mathfrak{g})$, $\mathfrak{g} =((g_{1},\ldots,g_{1}),(g_{2},\ldots,g_{2}),\ldots, (g_{k},\ldots,g_{k}))\in (G^{m_{block}})^{k}\simeq G^{m}$. Then the class $[c]$ is invariant under the action of
$G$, i.e. \ref{eq:0} holds.\end{lem}
\begin{proof}
First note that since \ref{eq:0} holds for $g\in H$ (indeed, it is well known that the inner action of $H$ induces a trivial action on $H^{2}(H,F^*)$) it is sufficient
to check the invariance of $[c]$ for the action of elements from $\mathfrak{g}_{\mathbf{set}}$ (a full set of $H$-coset representatives in $G$).

Suppose this is not the case. Note that by using elementary moves on the presentation of $A$ we may assume the tuple $\mathfrak{g}$ is ordered so that

\begin{enumerate}

\item
Equal cosets representatives are adjacent.

\item

Elements $g_i$ which normalize the cohomology class $[c]$ ``stand on the left of the tuple'' whereas $g_j$'s which do not normalize ``stand on the right'', that is there is $1 \leq t_{0} < k$ such that $g_{j}$ normalizes $[c]$ if $1 \leq j \leq t_{0}$ and $g_{j}$ does not normalize $[c]$ if $t_{0} < j \leq k$.

\end{enumerate}
Note that since $e$ normalizes $[c]$ we may assume $g_1=e$ whereas (by our ordering) $g_{k}$ does not normalize $[c]$.

Now, we know that the set of Binomial Identities of $F^{\gamma}H$ determines the cohomology class $\gamma$ (indeed, the Binomial identities generate the $T$-ideal of graded identities and the latter determines $\gamma$). Moreover, if $\alpha$ and $\gamma$ are \textit{different} cohomology classes of $H^{2}(H,F^{*})$ we have $Id_{H}(F^{\gamma}H) \nsubseteq Id_{H}(F^{\alpha}H)$. Let us fix for any such pair $(\alpha, \gamma)$ a binomial polynomial $B(\bar{\alpha}, \gamma)$ which is an $H$-graded identity of $F^{\gamma}H$ but a nonidentity of $F^{\alpha}H$. Furthermore, for any $\beta \in H^{2}(H,F^{*})$ we let $R(\beta)=\prod_{\gamma \neq \beta}(B(\bar{\beta}, \gamma)))$, where the binomial polynomials are taken with disjoint sets of variables implying $R(\beta)$ is multilinear. Note that since any twisted group algebra is a $H$-graded division algebra the polynomial $R(\beta)$ is a nonidentity of $F^{\beta}H$ but an $H$-identity of $F^{\gamma}H$ for any $\gamma\neq \beta$.

Now, we proceed as in [AH]. Let $p(X_{m^{2}_{block}}, Y_{m^{2}_{block}})$ denote the central Regev polynomial on $2m^{2}_{block}$ variables. Fixing $h \in H$, construct an $H$-graded polynomial $p_{h}(X_{m^{2}_{block}}, Y_{m^{2}_{block}})$ which is obtained from $p(X_{m^{2}_{block}}, Y_{m^{2}_{block}})$ by replacing one of its $X$ variables, say $x_{1}$, by a variable $x_{1,h}$ homogeneous of degree $h$, and all other variables $x_{i}, 2 \leq i \leq m^{2}_{block}$ and $y_{j}, 1 \leq j \leq m^{2}_{block}$ are replaced by variables of degree $e$. Finally, we replace all variable of $R(\beta)$ by Regev central polynomials. More precisely, we replace any variable in $R(\beta)$ of degree $h$ by a multilinear polynomial $p_{h}(X_{m^{2}_{block}}, Y_{m^{2}_{block}})$ (as usual, different variables of $R(\beta)$ are substituted by multilinear polynomial with disjoint sets of variables). We obtain a multilinear polynomial which we denote by $R(\beta, m_{block})$.

For $i=1,\dots,k$, consider the monomial of $e$-variables

$$Z_{i}=y_{0}^{(i)}x_{1}^{(i)}y_{1}^{(i)}\cdots x_{m_{\mathbf{block}}^{2}}^{(i)}y_{m_{\mathbf{block}}^{2}}^{(i)}$$
and let
$$Z=Z_{1}w_{1}Z_{2}\cdots w_{k-1}Z_{k},$$ where $w_{i}$
is a $g_{i}^{-1}h_{i}g_{i+1}-$graded variable and $h_{i} \in H$ to be chosen (in fact any element of $H$ will do). For $i=1,\ldots,k$, we insert in $Z$ the polynomial $R(g_{i}([c]), m_{block})$ on the right of $Z_{i}$. We let
$$
\hat{Z} = Z_{1}R(g_{1}([c]), m_{block})w_{1}Z_{2}R(g_{2}([c]), m_{block})\cdots w_{k-1}Z_{k}R(g_{k}([c]), m_{block}).
$$

Finally, we let $\Theta$ be the multilinear polynomial obtained by alternating the set of $e$-variables $X= \{x^{(i)}_j\}$ in $\hat{Z}$. We claim

\begin{enumerate}

\item

$\Theta$ is a $G$-graded nonidentity of $A$.
\item

All evaluations of the $G$-graded polynomial $\Theta$ on $A$ yield elements in $\hat{R}=span\{F^{[c]}H\otimes e_{r,s}^{(i,j)}\}$ where $(i,j)$ satisfies the following condition: $g_{1}^{-1}g_{i}$ and $g_{k}^{-1}g_{j}$ normalize the cohomology class $[c]$. Note that this is equivalent to
$$R(g_{i}([c]), m_{block}) = R(g_{1}([c]), m_{block})$$ and $$R(g_{j}([c]), m_{block}) = R(g_{k}([c]), m_{block}).$$

\item

$\hat{R}^{2}=0$.

\end{enumerate}

Note that the lemma follows at once from the claim. Indeed, if we duplicate the polynomial $\Theta$ with different variables and denote them by $\Theta_{1}$ and $\Theta_{2}$, we obtain two homogeneous $G$-graded nonidentities with different variables whereas their product is an identity.

To see the first part of the claim consider the evaluation of the monomials $Z_{i}$, $i=1,\ldots,k$ on the $i$th $e$-block of $A$. We then evaluate the polynomials $R(g_{i}([c]), m_{block})$, $i=1,\ldots,k$, on the same block of $A$ (the $i$th block). Finally, the variables $w_i$, $i=1,\ldots, k-1$ are evaluated at elements that bridge the $i$th and $i$th$+1$ diagonal blocks. By construction the value of $Z$ is nonzero. The polynomial $\Theta$ is obtained from $Z$ by alternating the variables of $X$. Since any nontrivial permutation of these variables yields a zero value (the values of $\{y^{i}_{j}\}$ and $\{x^{i}_{j}\}$ do not match) we get a nonzero evaluation of $\Theta$ on $A$ as desired.

For the second part of the claim note that by the alternation the elements of $X$ must be evaluated precisely on all $e$-elements of $A$ as long as the value is nonzero. Moreover, since the variables $y^{i}_{j}$ are also $e$-variables, the elements of $Z_{i}$ (and in particular the $x$'s variables of $Z_{i}$) must be evaluated on the same $e$-block of $A$, again as long as the evaluation is nonzero. But the presence of the polynomials $R(g_{j}([c]), m_{block})$ in $\Theta$ implies the evaluation of $\Theta$ on $A$ vanishes unless the monomial $Z_{j}$, any $j$, is evaluated on a $e$-block whose index is $\hat{j}$ for which $g_{\hat{j}}([c])=g_{j}([c])$. In particular the monomials $Z_{1}$ and $Z_{k}$ must be evaluated on $e$-blocks indexed by $\hat{1}$th and $\hat{k}$th respectively. This proves part $2$ of the claim. The $3$rd part of the claim follows from the fact that $\hat{1} < \hat{k}$. This completes the proof of the claim and hence of the lemma.

\end{proof}

In order to complete the proof of Theorem \ref{main theorem1} we need to show that if $A$ is a finite dimensional $G-$simple algebra with presentation $(H,[c],\mathfrak{g})$, where

\begin{enumerate}
\item

$H$ is normal,

\item

$\mathfrak{g} =((g_{1},\ldots,g_{1}),(g_{2},\ldots,g_{2}),\ldots, (g_{k},\ldots,g_{k}))\in (G^{m_{block}})^{k}\simeq G^{m}$

\item

$[c] \in H^{2}(H,F^{*})$ is $G$ invariant

\end{enumerate}
then $A$ is strongly verbally prime.

For the proof we shall replace the $3$rd condition with an apparently stronger condition $3'$ which we call the \textit{path property}. After introducing that property, we will show firstly that $1 + 2 +3'$ is sufficient for strongly verbally prime and secondly that conditions $1 + 2 +3$ and $1 + 2 +3'$ are equivalent. For future reference let us state the following.

\begin{thm} \label{main theorem1'}
Let $A$ be a finite dimensional $G$-simple algebra over $F$ with presentation $P_{A}=(H, \alpha, \mathfrak{g}=(g_1,\ldots,g_k))$. Then condition $1 +2 +3$ $($Theorem \ref{main theorem1}$)$ is equivalent to $1 +2 +3'$ below.

\begin{enumerate}
\item
The group $H$ is normal in $G$.

\item

The different cosets of $H$ in $G$ are represented exactly the same number of times in the $k$-tuple $\mathfrak{g}=(g_1,\ldots,g_{k})$. In particular the integer $k$ is a multiple of $[G:H]$.

\item[(3')]
 The \textit{path property}.

\end{enumerate}

\end{thm}

Let us introduce the notation we need.

\begin{defn} (Good permutations and pure polynomials)\label{Good and Pure}

For $g\in G$ denote
by $\overline{g}$ its image in $G/H$. We say the $G-$graded
multilinear monomial $Z_{\sigma}=x_{\sigma(1),t_{\sigma(1)}}\cdots x_{\sigma(n),t_{\sigma(n)}}$ is a \textit{good permutation}
of $Z=x_{1,t_{1}}\cdots x_{n,t_{n}}$
if
\begin{enumerate}
\item $(\mbox{deg}Z=)t_{1}\cdots t_{n}=t_{\sigma(1)}\cdots t_{\sigma(n)}(=\mbox{deg}Z_{\sigma})$.
\item For every $1\leq i\leq n$, $\overline{t_{1}\cdots t_{i}}=\overline{t_{\sigma(1)}\cdots t_{\sigma(\sigma^{-1}(i))=i}}$.
\end{enumerate}

A $G$-graded multilinear polynomial is said to be \textit{pure} if all its monomials are \textit{good permutations} of each other.
\end{defn}

\begin{note} \begin{enumerate}

\item Good permutation on monomials is an equivalence relation.
\item If all variables degrees
are in $H$, then the second condition says nothing. In that case \textit{good permutation} reduces to $Z,Z_{\sigma}$ having the same $G$-degree.
\end{enumerate}
\end{note}

Our first step is to show that it is sufficient to consider \textit{pure} polynomials when analyzing the $T$-ideal of $G$-graded identities of $A$. We assume the algebra $A$ is $G-$simple with presentation $(H,[c],\mathfrak{g})$, where $H$ is normal and

$\mathfrak{g} =((g_{1},\ldots,g_{1}),(g_{2},\ldots,g_{2}),\ldots, (g_{k},\ldots,g_{k}))\in (G^{m_{block}})^{k}\simeq G^{m}$.

\begin{lem}\label{decomposition of identities into Pure polynomials}

Let $p=\sum_{\sigma \in S_{n}}\alpha_{\sigma}Z_{\sigma}$ be a multilinear $G$-graded polynomial. Let $p=p_1+\ldots +p_q$ be the decomposition of $p$ into its \textit{pure} components. Then $p$ is a $G$-graded identity of $A$ if $($and only if$)$ $p_{i}$ is a $G$-graded identity of $A$ for any $i=1,\ldots,q$.

\end{lem}

\begin{proof}
Since $G$-homogeneous components are linearly independent subspaces of $A$ we may restrict our attention to monomials which have the same homogeneous degree so the first condition in the definition of \textit{good permutation} is satisfied. We therefore let all monomials of $p$ be homogeneous of degree $g$.
Since the polynomial $p$ is multilinear we can restrict evaluations to a spanning set of $A$ and in particular to elements of the form $b_{h, \{i,j\},E}=u_{h}\otimes e_{i,j} \otimes E$, where $h \in H$, $i,j \in \{1,\ldots, [G:H]\}$, $E \in M_{r}(F)$. Note that the homogeneous degree of $b_{h, \{i,j\},E}$ is determined by $u_{h}\otimes e_{i,j}$, namely $g_{i}^{-1}hg_{j}$, as the degree of $E$ is trivial.
Let $x_{g}$ be a variable of degree $g$. By our assumption, its admissible evaluations are elements $b_{h, \{i,j\},E}$ where $g=g_{i}^{-1}hg_{j}$. The following observations are important.
\begin{enumerate}
\item
For any $i\in \{1,\ldots, [G:H]\}$ there exist $h$ and $j$ such that $g=g_{i}^{-1}hg_{j}$.
\item
$h=h(i)$ and $j=j(i)$ are uniquely determined.

\end{enumerate}
This follows easily from the fact that $\{g_j\}$ is a full set of $H$-cosets representatives and $g_{i}g \in \cup_{j} Hg_{j}$.

Let $Z=x_{t_1}x_{t_2}\cdots x_{t_n}$. Following the above observations we choose any index $i \in \{1,\ldots,[G:H]\}$. This determines $h=h(i)$ and $j=j(i)$ in any possible value of $x_{t_1}$, that is $x_{t_1}$ is evaluated at $b_{h(i), \{i,j(i)\},E}$, any $E\in M_{r}(F)$, where $g_{i}^{-1}h(i)g_{j(i)} = t_1$ (i.e. the degree of $b_{h(i), \{i,j(i)\},E}$ is $t_1$). It follows that if we insist on nonzero evaluations of the monomial $Z$, the evaluation of $x_{t_2}$ must be of the form $b_{h', \{j(i),j'\},E'}$ and so on. We see that the choice of $i=i_1$ (for the evaluation of $x_{t_1}$) determines uniquely the triples $(i_1,j_1,h_1),\ldots,(i_n,j_n, h_n)$. Note on the other hand that a monomial $Z$ as above indeed admits nonzero evaluations with elements $b_1,\ldots,b_n$ where $b_\nu =b_{h_\nu, \{i_\nu,j_\nu\},E_{\nu}}$ for any choice $i_1 \in \{1,\ldots,[G:H]\}$ (for instance any evaluation of the form where $E_t$ is invertible will do).

Claim: Let $Z=x_{1,t_1}x_{2,t_2}\cdots x_{n,t_n}$ be a $G$-graded monomial as above and let
$$(i_1,j_1,h_1),\ldots,(i_n,j_n, h_n)$$ be $n$-triples such that the evaluation of $Z$  with $x_{t_\nu} \leftarrow b_\nu =b_{h_\nu, \{i_\nu,j_\nu\},E_{\nu}}$, $\nu=1,\ldots,n$, is nonzero. Let $Z_\sigma=x_{\sigma(1),t_\sigma(1)}x_{\sigma(2),t_\sigma(2)}\cdots x_{\sigma(n),t_\sigma(n)}$ be a monomial obtained by permuting the variables of $Z$ with $\sigma \in S_{n}$. Suppose the total $G$-degrees of $Z$ and $Z_{\sigma}$ coincide. Suppose also the evaluation above satisfies the following conditions

\begin{enumerate}

\item
The evaluation does not annihilates $Z_{\sigma}$

\item

$i_{\sigma(1)}=i_{1}$

\end{enumerate}
Then $Z$ and $Z_{\sigma}$ are \textit{good permutations} of each other. Consequently, $j_{\sigma(n)}=j_{n}$ and $e_{i_{\sigma(1)},j_{\sigma(1)}}\cdots e_{i_{\sigma(n)},j_{\sigma(n)}} = e_{i_{1},j_{1}}\cdots e_{i_{n},j_{n}}$.

Let us complete the proof of the lemma using the claim. Suppose $p_1(x_{t_1},x_{t_2}\cdots x_{t_n})$ is a nonidentity. We need to show $p(x_{t_1},x_{t_2}\cdots x_{t_n})$ is a nonindentity. Let $x_{t_\nu} \leftarrow b_\nu =b_{h_\nu, \{i_\nu,j_\nu\},E_{\nu}}$, $\nu=1,\ldots,n$ be a nonzero evaluation of $p_1$. The values of $p_1$ are contained in one block, say block $(i_0,j(i_0))$ of the block permutation matrix which corresponds to $g \in G$. On the other hand, by the claim, the evaluations on $p_s$, $s > 1$ are contained in blocks $(i, j(i))$, with $i \neq i_o$. This shows the evaluation of $p$ is nonzero and we are done.

\begin{proof} (claim)

Our setup implies the following:

\begin{enumerate}
\item

$j_1=i_2, j_2=i_3,\ldots,j_{n-1}=i_{n}$

\item

$j_{\sigma(1)}=i_{\sigma(2)}, j_{\sigma(2)}=i_{\sigma(3)}, \ldots, j_{\sigma(n-1)}=i_{\sigma(n)}$

\item

$i_{1}=i_{\sigma(1)}$, $j_{n}=j_{\sigma(n)}$

\item

$t_1=g_{i_1}^{-1}h_{i_1}g_{j_1}, t_2=g_{i_2}^{-1}h_{i_2}g_{j_2}, \ldots, t_n=g_{i_n}^{-1}h_{i_n}g_{j_n}$

\end{enumerate}

Since we are assuming the monomials $Z$ and $Z_{\sigma}$ have the same homogeneous degree in $G$, in order to prove the \textit{good permutation} property we need to show the following condition holds.

If $\sigma(l)=k$ then the values of the products of elements in $G$

$$
g_{i_1}^{-1}h_{i_1}g_{j_1}\cdot g_{i_2}^{-1}h_{i_2}g_{j_2}\cdots g_{i_{k-1}}^{-1}h_{i_{k-1}}g_{j_{k-1}}
$$
and

$$
g_{i_{\sigma(1)}}^{-1}h_{i_{\sigma(1)}}g_{j_{\sigma(1)}}\cdot g_{i_{\sigma(2)}}^{-1}h_{i_{\sigma(2)}}g_{j_{\sigma(2)}}\cdots g_{i_{\sigma(l-1)}}^{-1}h_{i_{\sigma(l-1)}}g_{j_{\sigma(l-1)}}
$$

coincide in $G/H$.

Indeed, ignoring the elements of $H$ and invoking the 1st and 2nd sets of equalities it is sufficient to show

$$
g_{i_1}^{-1}g_{j_{k-1}}=g_{i_{\sigma(1)}}^{-1}g_{j_{\sigma(l-1)}}.
$$

But $g_{i_1}^{-1} =  g_{i_{\sigma(1)}}^{-1}$ by (3) and so it remains to show $g_{j_{k-1}}=g_{j_{\sigma(l-1)}}$.
By $\sigma(l)=k$, (1) and (2) above we have $g_{j_{k-1}}=g_{i_k}=g_{i_{\sigma(l)}}= g_{j_{\sigma(l-1)}}$
and we are done. This proves the claim and completes the proof of the lemma.
\end{proof}
\end{proof}

Now we can introduce the \textit{path property}.

As usual we let $A$ be a $G$-simple algebra with presentation $(H,[c],\mathfrak{g})$, where $H$ is normal and $\mathfrak{g} =((g_{1},\ldots,g_{1}),(g_{2},\ldots,g_{2}),\ldots, (g_{k},\ldots,g_{k}))\in (G^{m_{block}})^{k}\simeq G^{m}$. Let $p(x_{t_1},\ldots,x_{t_n})$ be a $G$-graded \textit{pure} polynomial.
Write

$$p(x_{t_1},\ldots,x_{t_n}) = Z + \sum_{e \neq \sigma \in \Lambda(Z)}\alpha_{\sigma}Z_{\sigma}$$
where $Z=x_{t_1}\cdots x_{t_n}$ and $\Lambda(Z) = \{\sigma \in S_{n}: \sigma$ \textit{good permutation} of $Z$\}.

\begin{defn} (\textit{Path property}). Notation as in lemma \ref{decomposition of identities into Pure polynomials}.

\begin{enumerate}
\item

We say the polynomial $p(x_{t_1},\ldots,x_{t_n})$ satisfies the \textit{path property} if the following condition holds: If $p$ vanishes whenever the evaluation of $x_{t_1}$ has the form $u_{h(i_1)}\otimes e_{i_1,j(i_1)} \otimes E$ for some index $i_1 \in \{1,\ldots,[G:H]\}$ and $E \in M_{r}(F)$, then $p$ is a $G$-graded identity of $A$.

\item
We say the $G$-simple algebra $A$ satisfies the \textit{path property} if any \textit{pure} polynomial satisfies the \textit{path property}.

\end{enumerate}
\end{defn}

Roughly speaking, the condition says that if $p$ is annihilated by choosing an ``evaluation path'', then it vanishes for all $[G:H]$ paths and hence it is a $G$-graded identity.


\begin{prop}\label{equivalence strongly and 1+2+3'}
Let $A$ be a $G$-simple algebra with presentation $(H,[c],\mathfrak{g})$ as above (i.e. $H$ is normal and $\mathfrak{g} =((g_{1},\ldots,g_{1}),(g_{2},\ldots,g_{2}),\ldots, (g_{k},\ldots,g_{k}))\in (G^{m_{block}})^{k}\simeq G^{m}$.) Then $A$ is strongly verbally prime if and only if satisfies the path property.

\end{prop}

\begin{proof}
Suppose the \textit{path property} does not hold. Then there exists a $g$-homogeneous \textit{pure} polynomial $p$ which vanishes on all evaluations arising from one path, say $i_1 \in \{1,\ldots,[G:H]\}$ but is a nonidentity of $A$. For any $j \in \{1,\ldots,[G:H]\}$ there exists a suitable \textit{pure} polynomial $p_{j}$ obtained by changing the scalars coefficients of $p$ (depending on the evaluation path $j$) such that $p_j$ is clearly a nonidentitity and vanishes on evaluations corresponding to the $j$th path (see remark below). Taking disjoint variables on different $p_{j}$'s we see that the product $p_1\cdots p_{[G:H]}$ vanishes on all paths and hence is an identity whereas $p_j$, $j=1,\ldots, [G:H]$, is a nonidentity. This shows $A$ is not strongly verbally prime.

\begin{rem}
Since we don't really need the above direction details are omitted.
\end{rem}


For the converse suppose $A$ satisfies the \textit{path property}. We let $p$ and $q$ be multilinear $G$-graded homogeneous polynomial nonidentities of $A$ with disjoint sets of variables. We need to show $p\cdot q$ is a nonidentity.

\begin{claim} We may assume $p$ and $q$ are \textit{pure} polynomials

\end{claim}

\begin{proof}

To see this, let $p= p_{1}+ \cdots +p_{r}$ and $q = q_{1}+ \cdots + q_{s}$ be the decomposition of $p$ and $q$ into its \textit{pure} components. Observe that since $p$ and $q$ are defined on disjoint sets of variables, the decomposition $pq=\sum_{i,j}p_{i}\cdot q_j$ is precisely the decomposition of $pq$ into its \textit{pure} components. Now suppose the sufficiency in the proposition holds for \textit{pure} polynomials and let $p$ and $q$ be multilinear homogeneous nonidentities. We show $pq$ is a nonidentity of $A$. But by Lemma \ref{decomposition of identities into Pure polynomials} this reduces to \textit{pure} polynomials and the claim is proved.

\end{proof}

So let $p$ and $q$ be multilinear, nonidentities \textit{pure} polynomials where $p(x_1,\ldots,x_n) = \sum_{\sigma \in \Delta_{p}}\alpha_{\sigma}Z_{\sigma}$ and $q(y_1,\ldots,y_m) = \sum_{\tau \in \Delta_{q}}\beta_{\tau}W_{\tau}$. Choose $i \in \{1,\ldots,[G:H]\}$ and consider evaluations of $p$ of the form $u_{h} \otimes e_{k,s} \otimes E$ corresponding to the path determined by $i$. Since the monomials of $p$ are \textit{good permutations} of each other, for any $\sigma \in \Delta_{p}$ there is a scalar $\gamma_{\sigma} \in F^{*}$ such that all evaluations corresponding to the $i$th path annihilate $p$ if and only if the ungraded polynomial $p^{i} = \sum_{\sigma \in \Delta_{p}} \gamma_{\sigma}\alpha_{\sigma}Z_{\sigma}$ is a polynomial identity of $M_{r}(F)$. Now, we are assuming the algebra $A$ satisfies the \textit{path property} and hence, since $p$ is a nonidentity, for any path, say the path determined by $i \in \{1,\ldots,[G:H]\}$, there exists an evaluations of the form $u_{h} \otimes e_{k,s} \otimes E$ which does not annihilate $p = \sum_{\sigma}\alpha_{\sigma}Z_{\sigma}$ and hence $p^{i}$ is a nonidentity of $M_{r}(F)$. In order to show $pq$ is a nonidentity we choose an evaluation of $q$ which corresponds to the path $j=j(i)$ determined by the chosen path for $p$ and let $q^{j}$ be the corresponding polynomial. We obtain the product $pq$ is a graded identity if and only if $p^{i}q^{j}$ is an ungraded identity of $M_{r}(F)$. Invoking the \textit{path property} of $A$, we have the polynomial $q^{j}$ is an ungraded nonidentity of $M_{r}(F)$ and so, being both $p^{i}$ and $q^{j}$ nonidentities and $M_{r}(F)$ verbally prime, the product $p^{i}q^{j}$ is a nonidentity and we are done.
\end{proof}

In order to complete the proof the of Theorem \ref{main theorem1} it remains to show that the conditions $1 + 2 + 3$ and $1 + 2 + 3'$  (see paragraph before Definition \ref{Good and Pure}) are equivalent. In fact, since we have seen already that a finite dimensional $G$-simple algebra $A$ which satisfies the condition $1 + 2 + 3'$ is strongly verbally prime, it remains to show $1 + 2 + 3$ implies $1 + 2 +3'$.

At first, we will consider a simpler case, namely we will assume $A=F^{c}H\otimes M_{[G/H]}(F)$, where $H$ is normal in $G$,
$\beta$ is a $G$ invariant cocycle and $[G/H]$ stands for a transversal
of $H$ in $G$ containing $e$ (i.e. each element of $[G/H]$ appears
exactly once in $\mathfrak{g}$, the tuple which provides the elementary grading on $A$).
We need to show condition $3'$ holds.

Claim: It is sufficient to show the following

\begin{prop}\label{necessity of path property r=1}

Notation as above. For every $Z=x_{g_{1}}\cdots x_{g_{n}}$ and
every \textit{good permutation} $\sigma\in S_{n}$, there is a scalar $s\in F^{\times}$
such that $Z-sZ_{\sigma}\in Id_{G}(A)$ $(Z_{\sigma}=x_{g_{\sigma(1)}}\cdots x_{g_{\sigma(n)}})$.

\end{prop}

\begin{acknowledgment}
We are thankful to Darrell Haile for his contribution to the proof of this important proposition.
\end{acknowledgment}

Let us prove condition $3'$ assuming the proposition. Let $p(x_{g_1},\ldots,x_{g_n}) = \sum_{\sigma \Lambda}\beta_{\sigma}Z_{\sigma}$ be a \textit{pure} polynomial. Note that $\beta_{\sigma}$ may be zero. We need to show that if $p$ vanishes on one path, it vanishes on every path and hence it is a $G$-graded identity. Clearly, we may assume $\beta_{e} = 1$. We proceed by induction on the number $k_{0}$ of nonzero coefficients $\beta_{\sigma}$ in $p$ and so by way of contradiction we assume $p$ is a minimal counter example. Note that $k_{0} > 1$ since any multilinear monomial does not vanish on any path. Suppose $k_{0} \geq 3$ and let $p = Z + \beta_{\sigma_1}Z_{\sigma_1} + \beta_{\sigma_2}Z_{\sigma_2}+ \sum_{\tau \in \Lambda \setminus \{e, \sigma_{1}, \sigma_{2}\}} \beta_{\tau}Z_{\tau}$, $\beta_{\sigma_1}, \beta_{\sigma_2} \neq 0$. Let $s = s_{\sigma_1}$ be such that $Z-sZ_{\sigma}$ is a $G$-graded identity. It follows that $(\beta_{\sigma_1} -s)Z_{\sigma_1} +  \sum_{\tau \in \Lambda \setminus \{e, \sigma_{1}, \sigma_{2}\}} \beta_{\tau}Z_{\tau}$ doesn't satisfy the \textit{path property}. Contradiction to the minimality of $p$. So we need to show the path property is satisfied by \textit{pure} polynomials of the form $p = Z + \beta_{\sigma}Z_{\sigma}$ where $\beta_{\sigma}\neq 0$. But again, if $\beta_{\sigma} \neq s$, $p$ does not vanish on any evaluation path and the claim follows.

\begin{proof} (of Proposition)

Write $\hat{g}(i)$ for $[g_{1}\cdots g_{i}]$ and $\hat{g}(0)=e$. Here $[g]\in[G/H]$ is the corresponding element to $g\in G$.

Let $Z = x_{g_1}\cdots x_{g_n}$ and $Z = x_{g_{\sigma(1)}}\cdots x_{g_{\sigma(n)}}$ be \textit{good permutations} of each other. Then $g_{\sigma(1)}\cdots g_{\sigma(n)}=g_{1}\cdots g_{n}$
and
\begin{equation}
\hat{g}(\sigma(i))=\widehat{g_{\sigma}}(i):=[g_{\sigma(1)}\cdots g_{\sigma(i)}]\label{eq:1}
\end{equation}
 for $i=1,\ldots,n$. By setting $\sigma(0)=0$ we get equivalently
\begin{equation}
\hat{g}(\sigma(i)-1)=\widehat{g_{\sigma}}(i-1)\label{eq:2}
\end{equation}
 for $i=1,\ldots,n$.

Fix $t\in[G/H]$. For $i=1,\ldots,n$ we evaluate $x_{g_{i}}$ by $u_{[t\hat{g}(i-1)]g_{i}[t\hat{g}(i)]^{-1}}\otimes e_{[t\hat{g}(i-1)],[t\hat{g}(i)]}$
and obtain that $Z$, under this substitution, is equal to
\[
\left(\prod_{i=1}^{n}u_{[t\hat{g}(i-1)]g_{i}[t\hat{g}(i)]^{-1}}\right)\otimes e_{t,[t\hat{g}(n)]}.
\]
Next, let $\alpha_{i,t}\in H$ ($i=0,\ldots,n$) be such that
\[
[t\hat{g}(i)]=\alpha_{i,t}t[\hat{g}(i)].
\]
So
\begin{multline}
h_{i,t}:=[t\hat{g}(i-1)]g_{i}[t\hat{g}(i)]^{-1}= \\ = \alpha_{i-1,t}\cdot t\cdot\underset{f_{i}\in H}{\underbrace{\left([\hat{g}(i-1)]g_{i}[\hat{g}(i)]^{-1}\right)}}\cdot t^{-1}\cdot\alpha_{i,t}^{-1}
=\alpha_{i-1,t}f_{i}^{t}\alpha_{i,t}^{-1},
\label{eq:usefull}
\end{multline}
where
$\alpha_{0,t}=e$ and $\alpha_{\sigma(n),t}=[t\hat{g}(\sigma(n)][\hat{g}(\sigma(n)]^{-1}t^{-1}=[t\hat{g}(n)][\hat{g}(n)]^{-1}t^{-1}=\alpha_{n,t}$.

\begin{lem}
Every $f_{i}$ is independent of $t$ and we have the equality
$$
f_{1}^{t}\cdots f_{n}^{t}=f_{\sigma(1)}^{t}\cdots f_{\sigma(n)}^{t}.
$$
\end{lem}
\begin{proof}
By \ref{eq:1} and \ref{eq:2},
\[
\hat{g}(\sigma(i-1))=\widehat{g_{\sigma}}(i-1)=\hat{g}(\sigma(i)-1)
\]
for $i=1,\ldots,n$. Hence,
\begin{eqnarray*}
f_{\sigma(1)}^{t}\cdots f_{\sigma(n)}^{t} & = & \left(\prod_{i=1}^{n}[\hat{g}(\sigma(i)-1)]g_{\sigma(i)}[\hat{g}(\sigma(i))]^{-1}\right)^{t}\\
 & = & \left(\prod_{i=1}^{n}[\hat{g}(\sigma(i-1))]g_{\sigma(i)}[\hat{g}(\sigma(i))]^{-1}\right)^{t}=\left(g_{\sigma(1)}\cdots g_{\sigma(n)}\cdot[\hat{g}(n)^{-1}]\right)^{t}\\
 & = & (g_{1}\cdots g_{n}^{-1}\cdot[\hat{g}(n)^{-1}])^{t}=f_{1}^{t}\cdots f_{n}^{t}.
\end{eqnarray*}
\end{proof}
\begin{lem}
For $i=1,\ldots,n$
\begin{equation}
\alpha_{\sigma(i)-1,t}=\alpha_{\sigma(i-1),t}\label{eq:4}
\end{equation}

\end{lem}
As a result,
\begin{equation}
h_{\sigma(i),t}=\alpha_{\sigma(i-1),t}f_{\sigma(i)}^{t}\alpha_{\sigma(i),t}^{-1}\label{eq:5}
\end{equation}

\begin{proof}
We saw in the previous proof that $\hat{g}(\sigma(i-1))=\hat{g}(\sigma(i)-1)$
for $i=1,\ldots,n$. Thus,
\[
\alpha_{\sigma(i-1),t}=[tg(\sigma(i-1))][g(\sigma(i-1))]^{-1}t^{-1}=[tg(\sigma(i)-1)][g(\sigma(i)-1)]^{-1}t^{-1}=\alpha_{\sigma(i)-1,t}.
\]
 Thus,
\[
h_{\sigma(i),t}=\alpha_{\sigma(i)-1,t}f_{\sigma(i)}^{t}\alpha_{\sigma(i),t}^{-1}=\alpha_{\sigma(i-1),t}f_{\sigma(i)}^{t}\alpha_{\sigma(i),t}^{-1}.
\]
\end{proof}
Next, for $i=1,\ldots,n$
\[
u_{h_{i,t}}=u_{\alpha_{i-1,t}f_{i}^{t}\alpha_{i,t}^{-1}}=\frac{1}{c(\alpha_{i-1,t},f_{i}^{t})c(\alpha_{i-1,t}f_{i}^{t},\alpha_{i,t}^{-1})}\cdot u_{\alpha_{i-1,t}}\cdot u_{f_{i}^{t}}\cdot u_{\alpha_{i,t}^{-1}}.
\]
As a result,
\[
u_{h_{1,t}}\cdots u_{h_{n,t}}=\left(\prod_{i=1}^{n}\frac{1}{c(\alpha_{i-1,t},f_{i}^{t})c(\alpha_{i-1,t}f_{i}^{t},\alpha_{i,t}^{-1})}\right)\left(\prod_{i=1}^{n}c(\alpha_{i,t},\alpha_{i,t}^{-1})\right)\cdot u_{f_{1}^{t}}\cdots u_{f_{n}^{t}}\cdot u_{\alpha_{n,t}^{-1}}.
\]
Using \ref{eq:5} and \ref{eq:usefull} we compute

\bigskip

$u_{\sigma(h_{1,t})}\cdots u_{\sigma(h_{n,t})}=$

$$\left(\prod_{i=1}^{n}\frac{1}{c(\alpha_{\sigma(i-1),t},f_{\sigma(i)}^{t})c(\alpha_{\sigma(i-1),t}f_{\sigma(i)}^{t},\alpha_{\sigma(i),t}^{-1})}\right)\left(\prod_{i=1}^{n}c(\alpha_{\sigma(i),t},\alpha_{\sigma(i),t}^{-1})\right)
\cdot u_{f_{\sigma(1)}^{t}}\cdots u_{f_{\sigma(n)}^{t}}\cdot u_{\alpha_{n,t}^{-1}}.
$$

\bigskip

Hence, for every $t\in[G/H]$, we have
\[
u_{h_{1,t}}\cdots u_{h_{n,t}}=s(f_{1}^{t},\ldots,f_{n}^{t})\cdot u_{\sigma(h_{1,t})}\cdots u_{\sigma(h_{n,t})},
\]
where $s(f_{1}^{t},\ldots,f_{n}^{t})\in F^{\times}$ satisfies
\[
u_{f_{1}^{t}}\cdots u_{f_{n}^{t}}=s(f_{1}^{t},\ldots,f_{n}^{t})\cdot u_{f_{\sigma(1)}^{t}}\cdots u_{f_{'\sigma(n)}^{t}}.
\]
However, since we assumed that $[c]$ is invariant under $G$, we
must have that
\[
s(f_{1}^{t},\ldots,f_{n}^{t})=s(f_{1},\ldots,f_{n}).
\]

This completes the proof of the proposition and hence the implication $1 +2 +3 \Rightarrow 1 +2 +3'$ in case $r=1$.
\end{proof}

\section{\label{sec:division-algebras}$G-$graded division algebras}

In this section we prove

1) Theorem \ref{main theorem2}.

2) $1 + 2 + 3 \Rightarrow 1 + 2 +3'$ where $r$ is arbitrary (Theorem \ref{main theorem1'}).

3) Theorem \ref{nondisjoin variables}.

\begin{proof} (of \ref{main theorem2})
One direction is easy. Suppose such $B$ exists. Since $A$ and $B$ have the same $G$-graded identities it is sufficient to show that products of multilinear nonidentities of $B$ with disjoint sets of variables is a nonidentity of $B$. But this is clear since $p$ and $q$ have nonzero (invertible) evaluations on $B$ and hence also their product $pq$.

We proceed now to show that such a $B$ exists provided the algebra $A$ is $G$-graded strongly verbally prime. We will show that the corresponding \textit{generic} $G$-simple algebra is a $G$-graded division algebra.

We consider an affine relatively free $G$-graded algebra $U_{G} = F\langle X_{G} \rangle /Id_{G}(A)$ corresponding to $A$ with $Id_{G}(U_{G}) = Id_{G}(A)$ (it is known that such algebra exists (see subsection 7.1, \cite{AB})). Here $X_{G}$ is a finite set of $G$-graded variables. We claim $U_{G}$ is semisimple, that is the Jacobson radical $J=J(U_{G})$ is zero. Indeed, we know by Braun-Kemer-Razmyslov theorem $J$ is nilpotent. Moreover, the Jacobson radical is $G$-graded (see \cite{Cohen-Montgomery}) and hence this would contradict the strongly verbally prime condition on $A$ unless $J=0$.

Let $Z(U)$ denote the center of $U$ and $Z(U)_{e}$ the central $e$-homogeneous elements. Note: the center $Z(U)$ is not graded in general.

\begin{lem}

1) There exist $e$-homogeneous central polynomials, that is $Z(U)_{e}\neq 0$.

2) $Z(U)_{e}$ has no nontrivial zero divisors as elements of $U$.

\end{lem}

\begin{proof}
Both statements were proved in \cite{Karasik} (see Theorem $3.12$, Lemma $5.4$). Nevertheless, we recall the proof of the second statement since a similar argument will be used later. Let $p$ be a central $e$-homogeneous nonidentity and $q$ an arbitrary homogeneous nonidentity. Note that we are not assuming further conditions on $p$ and $q$, and in particular no assumption was made on the sets of variables of $p$ and $q$. We claim, we may assume that $p$ and $q$ are \textit{multihomogeneous} (each variable appears in any monomial with the same multiplicity). To see this put an ordering on the variables which appear either in $p$ or $q$. Counting variables multiplicities we order lexicographically the monomials in $p$ and $q$ respectively. Let $p=p_1+\ldots +p_r$ and $q=q_1+\dots +q_s$ be the decomposition of $p$ and $q$ into its multihomogeneous constituents where $p_1 < \ldots < p_r$ and $q_1 < \ldots < q_s$. Suppose $pq$ is a $G$-graded identity. It follows that its homogeneous constituent are $G$-graded identities and in particular $p_1q_1$ is a $G$-graded identity. This proves the claim.

Let $\mathfrak{V}_{p}$ and $\mathfrak{V}_{q}$ denote the set of variables appearing in $p$ or $q$ respectively. Let $\mathfrak{V} = \mathfrak{V}_{p} \cup \mathfrak{V}_{q}$ and let $n =card(\mathfrak{V})$. An evaluation of $p$ (or $q$) is a choice of an element in $A^{n} = F^{d}$, where $d=\dim_{F}(A)\times n$. Let $\mathfrak{A}_{p} = \{z \in A^{n}: p(z) \neq 0\}$ and $\mathfrak{B}_{q} = \{z \in A^{n}: q(z) = 0\}$. Note that $\mathfrak{A}_{p}$ is an open nonempty (and hence dense) subset of $A^{n}$ in the Zariski topology whereas $\mathfrak{B}_{q}$ is a closed set $\neq A^{n}$. Since $p$ take values in $Z(A)_{e} = F$, $pq$ being an identity is equivalent to $\mathfrak{A}_{p} \subseteq \mathfrak{B}_{q}$ which is impossible. This proves the second part of the lemma.
\end{proof}

Consider the central localization $\mathcal{A}_{G}=(Z(U)_{e}-\{0\})^{-1}U_{G}$. We obtain that $Z(\mathcal{A}_{G})_{e} = Frac(Z(U)_{e})$ is a field and we denote it by $K$.

Invoking \cite{Karasik}, Corollary 4.7 (Posner's theorem for $G$-graded algebras), we obtain a finite dimensional algebra $G$-simple over $K$. We need to show the $e$-component $(\mathcal{A}_{G})_e$ is an (ordinary) division algebra. Since the $e$-component is finite dimensional over $K$ we need to show there are no nontrivial zero divisors in $(\mathcal{A}_{G})_{e}$. Suppose this is not the case and let $p$ and $q$ be $e-$polynomials, nonidentities, whose product is zero (obtained by clearing denominators).

\begin{lem}
We may assume the polynomials $p$ and $q$ are multihomogeneous.

\end{lem}
\begin{proof}
Same as above.
\end{proof}

Recall that $\{u_{e} \otimes e_{i,i} \otimes E\}_{i \in \{1,\ldots, [G:H]\}, E \in M_{r}(F)}$ is a spanning set of the $e$-homogeneous component of $A$. We refer to
$A_{e,i}=span_{F}\{u_{e} \otimes e_{i,i} \otimes E\}$ as the $i$th component of $A_{e}$, $i=1,\ldots,[G:H]$.

Claim: Evaluations of any nonidentity $e$-polynomial on the algebra $A$ yield nonzero values on \textit{every} component of $A_{e}$. More precisely, if we denote by $P_{i}$ the projection $A_{e} \rightarrow A_{e,i}$, then for any $i$ the evaluation set of $P_{i}p$ on $A$ is nonzero.

Let us complete the proof of the Theorem using the claim. We let $\mathfrak{A}_{p,i} = \{z \in A^{n}: P_{i}p(z) \neq 0\}$  and  $\mathfrak{B}_{q,i} = \{z \in A^{n}: P_{i}q(z) = 0\}$. As $\mathfrak{A}_{p,i}$ is nonempty and open, $\mathfrak{B}_{q,i} \neq A^{n}$ and closed, we have $\mathfrak{A}_{p,i} \nsubseteq \mathfrak{B}_{q,i}$. This shows $P_{i}pq=P_ip\cdot P_iq: A^{n} \rightarrow M_{r}(F)$ is a nonzero polynomial function and we are done.

We turn now to the proof of the claim. We'll see that basically the claim follows from the path property which is known to hold for $A$ by Proposition \ref{equivalence strongly and 1+2+3'}.

Suppose first $p$ is $e$-homogeneous and pure. We know that if $p$ does not vanish on a certain path, say the $i$-th path, it does not vanish on any path $j \in \{1,\ldots, [G:H]\}$. Furthermore, we saw that different paths yields values on different diagonal blocks and so the claim is clear in this case. Suppose now $p$ is multilinear and let $p = p_1 + \ldots + p_q$ be the decomposition of $p$ into its pure components. Let $Z$ be a monomial of $p_1$ say and suppose it does not vanish on a certain path evaluation. Let $Z_{\sigma}$ be a monomial in $p$ which obtained from $Z$ by the permutation $\sigma \in S_{n}$. If $Z_{\sigma}$ is a good permutation of $Z$ then they are constituents of the same pure polynomial, the case we considered above. Otherwise, applying the claim in the proof of Lemma \ref{decomposition of identities into Pure polynomials} we know $Z$ and $Z_{\sigma}$ get values on different blocks (this includes the case the evaluation annihilates $Z_{\sigma}$). This establishes the claim in case $p$ is multilinear.

Finally, we suppose $p$ is multihomogeneous and proceed by induction on its multihomogeneous degree.
Suppose $p(x_{1},x_{2}\ldots,x_{n})$ is multihomogeneous where variables $x_{1},\ldots,x_{n}$ appear in degree $k_1,\ldots,k_n$ respectively. Without loss of generality we may assume $k_1 \geq k_2 \geq \cdots \geq k_n$. Clearly we may view the tuple $(k_1,\ldots,k_n)$ as a partition of $m=\sum_{s}k_{s}$. Put $x_{1}=t_{1}+\cdots +t_{k_1}$ and let
$$\bar{p}(t_1,\ldots,t_{k_1},x_{2},\ldots,x_{n}) = p((t_1+\ldots + t_{k_1}),x_{2},\ldots,x_{n}).$$

Since $P_{i}p(x_1,x_2,\ldots,x_n)=0$, we have that $P_{i}\bar{p}(t_1,\ldots,t_{k},x_{2},\ldots,x_{n})=0$ and hence also $P_{i}\bar{p}(t_{j},x_{2},\ldots,x_{n})=0$, for $j=1,\ldots,k_1$. It follows that
$$P_{i}(\bar{p}(t_1,\ldots,t_{k_1},x_{2},\ldots,x_{n}) - \sum_{j=1}^{k_1}p(t_{j},x_{2},\ldots,x_{n})) =0$$
where $\bar{p}(t_1,\ldots,t_{k},x_{2},\ldots,x_{n}) - \sum_{j}p(t_{j},x_{2},\ldots,x_{n}))$ is a sum of multihomogeneous polynomials of degree $m$ whose multihomogeneous degrees are strictly smaller than $(k_1,\ldots,k_n)$. The proof is completed by induction.

Thus, we have proved the algebra $\mathcal{A}_{G}=(Z(U)_{e}-\{0\})^{-1}U_{G}$ is a finite dimensional $G$-graded division algebra over the field $K = Z(\mathcal{A}_{G})_{e}$. It remains to show the algebra $\mathcal{A}_{G}$ is a twisted form of $A$. But this is clear since by construction $Id_{G}(\mathcal{A}_{G})= Id_{G}(A) = \Gamma$ and finite dimensional $G$-graded simple algebras over an algebraically closed field are determined up to a $G$-graded isomorphism by the $T$-ideal of $G$-graded identities.

\end{proof}

We still owe the reader the proof of the implication $1 + 2 +3 \Rightarrow 1 + 2 + 3'$ in the general case, namely in case $A=F^{c}H\otimes M_{[G/H]}(F)\otimes M_{r}(F)$, $r \geq 1$.
Interestingly, we complete the proof of that implication ($2 \rightarrow 3$ below) applying Theorem \ref{main theorem2}. For convenience, we summarize our main results in the following Theorem.

\begin{thm}

Let $\Gamma \subset F\langle X_{G} \rangle$ be a $G$-graded $T$-ideal which contains a Capelli polynomial. The following conditions are equivalent:

\begin{enumerate}
\item
$\Gamma$ is strongly verbally prime.

\item

There exists a finite dimensional algebra $G$-simple algebra $A=F^{c}H\otimes M_{[G/H]}(F)\otimes M_{r}(F)$ which satisfies conditions $1 + 2 + 3$ of Theorem \ref{main theorem1} such that $Id_{G}(A) = \Gamma$.

\item

There exists a finite dimensional algebra $G$-simple algebra $A=F^{c}H\otimes M_{[G/H]}(F)\otimes M_{r}(F)$ which satisfies conditions $1 + 2 + 3'$ of Theorem \ref{main theorem1'} such that $Id_{G}(A) = \Gamma$.

\item

There exists a field $K \supseteq F$ and a finite dimensional $G$-graded division algebra $B$ over $K$ with $Id_{G}(B) = \Gamma$.

\end{enumerate}

Furthermore
\begin{enumerate}

\item
The $G$-simple algebra $A$ in $(2)$ and $(3)$ is uniquely determined.

\item
The finite dimensional $G$-graded division algebra $B$ $($in $3$$)$ is uniquely determined in the following sense: Any two such algebras are twisted forms of each other, that is, if $B'$ is a finite dimensional $G$-graded division algebra over $K'$ with $Id_{G}(B') = \Gamma$, then there exists a field extension $E$ of $K$ and $K'$ such that $B\otimes_{K}E \cong B'\otimes_{K'}E$ as $G$-graded algebras.

\end{enumerate}
\end{thm}

\begin{proof}

Note that we have shown that $1 \rightarrow 2$ in the first part of section \ref{sec:strongly-verbally-prime} (Lemma \ref{lem:-same_block_size}, Proposition \ref{normal subgroup} and Lemma \ref{invariance of cohomology}). The equivalence $3 \leftrightarrow 1$ is stated and proved in Proposition \ref{equivalence strongly and 1+2+3'}. Above in this section, we showed $1 \leftrightarrow 4$ and in last part of section \ref{sec:strongly-verbally-prime} (Proposition \ref{necessity of path property r=1})  we showed $2 \rightarrow 3$ in case $r=1$. Let $A=F^{c}H\otimes M_{[G/H]}(F)\otimes M_{r}(F)$ and suppose it satisfies conditions $1 + 2 + 3$ of Theorem \ref{main theorem1}. Clearly, the proof will be completed if we show that $A$ admits a twisted form which is a $G$-graded division algebra. Consider the algebra $A_{0} = F^{c}H\otimes M_{[G/H]}(F)$. Clearly it satisfies $1 + 2 + 3$ of Theorem \ref{main theorem1} and since here $r=1$ it admits a finite dimensional $G$-graded twisted form $B_{0}$ over a field $K_{0}$ which is a $G$-division algebra. All we need to do now is to find a field $K$ and $G$-graded division algebra $B$ over some field $K$ which is a twisted form of $A$. Since $B_{0}$ is a $G$-graded division algebra over $K_{0}$ and since its $e$-component is commutative, we have that $L_{0} = (B_{0})_{e}$ is a finite field extension of $K_{0}$. Clearly, we may extend scalars to $K_{0}(t,s)$ (the rational function field over $K_{0}$ on ideterminates $t$ and $s$) and obtain a $G$-graded division algebra $(B_{0})_{K_{0}(t,s)}$ which is also a twisted form of $A_{0}$. It remains to prove there exists an ungraded division algebra $D$ of index $r$ over $K_{0}(t,s)$ which remains a division algebra after extending scalars to $L_{0}(t,s)$. Since $F$ is algebraically closed of characteristic zero, it contains a primitive $r$th root of unity $\mu_{r}$ and so we may consider the symbol algebra   $(t,s)_{r} = \langle u^{r}=t, v^{r}=s, vu = \mu_{r}uv \rangle$ over $K_{0}(t,s)$. It is easy to prove that $((t,s)_{r})_{L_{0}(t,s)}$ is a division algebra and we are done. Details are left to the reader.

Finally, for the uniqueness (second part of the theorem) we invoke once again the fact that any two $G$-graded simple algebras over an algebraically closed field $F$ of characteristic zero are $G$-graded isomorphic if and only if they satisfy the same $G$-graded identities (see \cite{AljHaile}).

\end{proof}

We close the paper with the proof of Theorem \ref{nondisjoin variables}.

\begin{proof}
Suppose $A$ is a finite dimensional $G$-graded division algebra over a field $K$ and let $p$ and $q$ be $G$-homogeneous polynomials such that $p$ and $q \notin Id_{G}(A)$. If $pq \in Id_{G}(A)$ we have that $\mathfrak{A}_{p} = \{z \in A^{n}: p(z) \neq 0\} \subseteq \mathfrak{B}_{q} = \{z \in A^{n}: q(z) = 0\}$. But this is impossible since $\mathfrak{A}_{p}$ is nonempty and open whereas $\mathfrak{B}_{p} \neq A^{n}$ and closed. For the opposite direction note that our condition on polynomials $p$ and $q$ implies $A$ is $G$-prime and hence by Theorem \ref{classification of G-graded prime} $A$ is $G$-graded PI equivalent to a finite dimensional $G$-simple algebra. On the other hand the condition on $p$ and $q$ implies $A$ is strongly $G$-graded verbally prime. Invoking Theorem \ref{main theorem1} and Theorem \ref{main theorem1'} the result follows.

\end{proof}

\end{document}